\documentclass[a4paper,oneside, reqno]{amsart}
\usepackage{mathrsfs}
\usepackage{amsfonts}
\usepackage{amssymb}
\usepackage{amsxtra}
\usepackage{dsfont}
\usepackage{amsthm}
\usepackage[colorlinks=true,linkcolor=blue]{hyperref}%
\usepackage[T1]{fontenc}
\usepackage[utf8]{inputenc}
\usepackage[english]{babel}
\usepackage{amsmath} 
\usepackage{bm}
\usepackage{graphicx}
\usepackage{hyperref}
\usepackage{color}

\usepackage{hyperref}
\usepackage{stmaryrd}
\usepackage{listings}

\usepackage{enumitem}

\numberwithin{equation}{section}
\usepackage{url}
\newtheorem{Thm}{Theorem}[section]

\newtheorem{Cor}[Thm]{Corollary}
\newtheorem{Lem}[Thm]{Lemma}
\newtheorem{Prop}[Thm]{Proposition}
\theoremstyle{definition}
\newtheorem{Rem}[Thm]{Remark}      
\theoremstyle{definition}
\newtheorem{Defn}[Thm]{Definition}
\newtheorem{Ex}[Thm]{Example}
\newtheorem{Fact}[Thm]{Fact}
\theoremstyle{definition}
\newtheorem*{Prob*}{Problem}
      
\theoremstyle{definition}

\newcommand{\Z}{\mathbb{Z}}
\newcommand{\D}{\mathbb{D}}
\newcommand{\kr}{\mathds{kr}}
\newcommand{\N}{\mathbb{N}}
\newcommand{\T}{\mathbb{T}}
\newcommand{\R}{\mathbb{R}}
\newcommand{\C}{\mathbb{C}}
\DeclareMathOperator{\Sym}{Sym}
\DeclareMathOperator{\supp}{supp}

\title[Discrete dyadic maximal function over $\ell^q$ balls: dimension-free estimates]
{Remark on dimension-free estimates for discrete maximal functions over $\ell^q$ balls: small dyadic scales}

\author{Jakub Niksiński}
\address[Jakub Niksiński]{Institute of Mathematics\\
	University of Wroclaw\\
	Plac Grunwaldzki 2\\
	50-384 Wrocław\\
	Poland}
\email{trolek1130@gmail.com}

\begin{document}

\selectlanguage{english}

\begin{abstract}
We give a dimension-free bound on $\ell^p(\Z ^d)$, $p \in [2, \infty]$ for the discrete Hardy-Littlewood maximal operator over the $\ell^q$ balls in $\Z ^d$ with small dyadic radii. Our result combined with the work of Kosz, Mirek, Plewa, Wróbel gives dimension-free estimates  on $\ell^p(\Z^d)$, $p \in [2, \infty]$ for the discrete dyadic Hardy-Littlewood maximal operator over $\ell^q$ balls for $q \geq 2$.
\end{abstract}


\maketitle



\section{Introduction}
Let $G$ be a convex, bounded, closed symmetric subset of $\mathbb{R}^d$ with non-empty interior, we will call such $G$ a symmetric convex body. Natural examples are the $\ell^q$ balls:
\[
B^q= \Big\lbrace x\in \mathbb{R}^d : \| x \|_{q}=\Big( \sum_{i=1}^d |x_i|^q \Big)^{1/q} \leq 1 \Big\rbrace \ \text{for} \  q \in [1, \infty ), 
\]
\[
B^{\infty}= \Big\lbrace x \in \mathbb{R}^d : \|x \|_{\infty}= \max_{1 \leq i \leq d} |x_i| \leq 1 \Big\rbrace.
\]
With each symmetric convex body one can associate the corresponding Hardy-Littlewood averaging operator. For any $t>0$ and $x \in \mathbb{R}^d$ we define
\[
M_t^G f(x)= \frac{1}{|t \cdot G|} \int_{t \cdot G} f(x-y) \ dy, 
\]
for a locally integrable function $f$, where $t \cdot G= \lbrace tx : x \in G \rbrace$ and $|t \cdot G|$ denotes its Lebesgue measure. Now for any $p$ let $C_p(d,G)>0$ be the smallest number such that the following inequality
\[
 \| \sup_{t>0} |M_t^G f| \|_{L^p(\mathbb{R}^d)} \leq C_p(d,G) \| f \|_{L^p(\mathbb{R}^d)}
\]
holds for every $f \in L^p(\mathbb{R}^d)$. It is well known that $C_p(d,G)< \infty$ for all $p \in (1, \infty]$ and all symmetric convex bodies $G$. 
\par
In 1980s dependency of $C_p(d,G)$ on dimension $d$ has begun to be studied. Various results were obtained, but as of now the major conjecture in this topic is still open, namely that $C_p(d,G)$ can be bounded from the above by a number independent of set $G$ and dimension $d$ for each fixed $p \in (1, \infty]$. We recommend the survey article \cite{HL-cont} for a detailed exposition of the subject, which contains results that we skipped mentioning.
\par 
Similar questions can be considered for the discrete analogue of the operator $M_t^G$. For every $t>0$ and every $x \in \Z ^d $ we define the discrete Hardy-Littlewood averaging operator
\[
\mathcal{M}_t^Gf(x)= \frac{1}{| (t \cdot G) \cap \Z^d|} \sum_{ y \in t \cdot G \cap \Z^d} f(x-y),
\]
where $f: \Z^d \to \C$, $| (t \cdot G) \cap \Z^d|$ is the number of elements of the set $(t \cdot G) \cap \Z^d$. Similarly as before we define $\mathcal{C}_p(d,G)>0$ to be the smallest number such that the following inequality
\[
 \| \sup_{t>0} |\mathcal{M}_t^G f| \|_{\ell^p(\mathbb{Z}^d)} \leq \mathcal{C}_p(d,G) \| f \|_{\ell^p(\mathbb{Z}^d)}
\]
holds for every $f \in \ell^p( \Z ^d)$. Using similar methods as in the continuous case one can show that for every $p \in (1, \infty]$ and symmetric convex body we have $\mathcal{C}_p(d,G) < \infty$. 
\par
What about the dependency of $\mathcal{C}_p(d,G)$ on the dimension $d$?
We can ask similar questions as in the continuous setup, yet it turns out that the situation is much more delicate. Indeed in \cite{cubes} the authors constructed a family of ellipsoids $E(d) \subseteq \mathbb{R}^d$, with the property that for each $p \in (1, \infty]$ there exists $C_p>0$ such that for every $d \in \N$ we have
\[
\mathcal{C}_p(d, E(d)) \geq C_p ( \log d)^{1/p}.
\]
This means that if we want to establish dimension-free estimates for $\mathcal{C}_p(d, G)$ we need to restrict ourselves to specific sets $G$, which contain more symmetries; one of the simpler options are $B^q$ sets. \par Define $\mathbb{D}=\lbrace 2^n : n \in \mathbb{N}_0 \rbrace$ to be the set of dyadic numbers.
Literature in the discrete setting is not as fruitful, nevertheless there are some papers and  positive results in this regard, for example:
\begin{itemize}
    \item In \cite{cubes} it was proved that for every $p \in ( \frac{3}{2}, \infty ]$ there exists a constant $C_p>0$ such that for every $d \in \N$ and every $f \in \ell^p(\Z ^d)$ we have
    \[
    \Big\| \sup_{t>0} |\mathcal{M}_t^{B^{\infty}} f| \Big\|_{\ell^p(\Z ^d)} \leq C_p \| f \|_{\ell^p(\Z ^d)}.
    \]
    This result is almost as strong as those in the continuous case. In the case of sets $B^q$ for $q \neq \infty$ authors of papers \cite{cubes}, \cite{balls}, \cite{Bq}  could only obtain weaker conclusions. 
    \item In \cite{balls} it was proved that there exists $C>0$ such that for every $p \in [2, \infty)$, $d \in \N$ and every $f \in \ell^p(\Z ^d)$ we have
    \[
    \Big\| \sup_{t \in \D} |\mathcal{M}_t^{B^2} f| \Big\|_{\ell^p(\Z ^d)} \leq C \| f \|_{\ell^p(\Z ^d)}.
    \]
    \item In \cite{Bq} by extending methods of \cite{balls}  it was proved that for any $q \in (2, \infty)$ there exists constant $C(q)>0$ 
    such that for every $p \in [2, \infty)$, $d \in \N$ and every $f \in \ell^p(\Z ^d)$ we have
    \[
    \Big\| \sup_{t \in \D, t \geq d^{1/q}} |\mathcal{M}_t^{B^q} f| \Big\|_{\ell^p(\Z ^d)} \leq C(q) \| f \|_{\ell^p(\Z ^d)}.
    \]
     The paper \cite{Bq} did not cover the range $t<d^{1/q}$ nor $q<2$.
    \item  More recently in \cite{NW}, the author in collaboration with B. Wróbel has shown that for any $\varepsilon>0$ and there is a constant $C_{\varepsilon}>0$ such that for all $p \in [2,\infty)$, $d \in \N$, and $f \in \ell^p(\Z^d)$ we have
    \[
    \Big\| \sup_{0 \le t \le d^{\frac{1}{2}-\varepsilon}} |\mathcal{M}_t^{B^2}f| \Big\|_{\ell^p(\Z^d)} \le C_{\varepsilon} \|f \|_{\ell^p(\Z^d)}.
    \]
    
\end{itemize}
\par
In this paper we will prove the following result.
\begin{Thm} \label{thm:1.1}
 There exists $C>0$, such that for every $q \geq 1$, $p \in [2,\infty]$, $d \in \mathbb{N}$, and  $f \in \ell^p(\Z ^d)$ we have \\
\[ \Big\| \sup_{t \in \mathbb{D}, t \leq d^{1/q}} |\mathcal{M}_t^{B^q} f| \Big\|_{\ell^p(\mathbb{Z}^d)} \leq C \| f \|_{\ell^p(\mathbb{Z}^d)}. \]

\end{Thm}
The result above for $q=1$ gives a bound for supremum over $t \in \D, t \leq d$, which generalizes the main theorem of \cite{ja}. Moreover Theorem \ref{thm:1.1} for $q \ge 2$ combined with \cite[Theorems 2 and 3]{Bq}
leads to the following corollary.
\begin{Cor}
    \label{cor:1.2}
For every $q \geq 2$ there exists $C(q)>0$, such that for any  $p \in [2,\infty]$, $d \in \mathbb{N}$ and $f \in \ell^p(\Z ^d)$ we have \\
\[ 
\Big\| \sup_{t \in \mathbb{D}} |\mathcal{M}_t^{B^q} f| \Big\|_{\ell^p(\mathbb{Z}^d)} \leq C(q) \| f \|_{\ell^p(\mathbb{Z}^d)}. 
\]
\end{Cor}
Methods of \cite{Bq} can even give a $C(q)$ uniform with respect to $q \ge 2$, however we don't claim any results in this regard and will not comment further on this matter. \par 
Our proof of Theorem \ref{thm:1.1} uses methods of \cite[Section 3]{balls} in a more streamlined version, adapting them if necessary. The fact that these methods still lead to new conclusions was apparently missed by the authors of \cite{Bq}. Case $q \in [1,2)$ leads to some new difficulties, which require ad hoc lemmas such as Lemma \ref{lem:2.4} or Lemma \ref{sum_z}.  
\subsection{Notation} \label{sec1.1}
 $\mathbb{N}=\lbrace 1,2,... \rbrace$ will denote the set of positive integers and $\mathbb{N}_0=\mathbb{N} \cup \lbrace 0 \rbrace$ will denote the set of non-negative integers. $\mathbb{D}=\lbrace 2^n : n \in \mathbb{N}_0 \rbrace$ is the set of dyadic integers. For $N \in \N$ we define
$[N]= \lbrace 1,2,...,N \rbrace $. 
 For $q \ge 1, t \ge 0$ let
\[
B_t^q= \{ x \in \mathbb{R}^d : \sum_{i=1}^d |x_i|^q \leq t^q \},
\]
notice that the set $B_t^q$ depends also on the parameter $d$.

 We define $e(x)=e^{2\pi i x}$ for any $x \in \mathbb{R}^d$. 
 We use the standard scalar product on $\mathbb{R}^d$
\[
 x \cdot y= \sum_{k=1}^d x_ky_k,
\]
where $x,y \in \mathbb{R}^d.$ 
 For $p \in [1,\infty)$, $f: \mathbb{Z}^d \to \mathbb{C}$ and $x \in \mathbb{R}^d$ let
\[
\|x \|_{p}= \Big( \sum_{i=1}^d |x_i|^p \Big)^{1/p}, \quad \| f\|_{\ell^p(\Z^d)}= \Big( \sum_{x \in \Z^d} |f(x)|^p \Big)^{1/p}
\]
and
\[
\ell^p(\Z^d)= \{ g: \Z^d \to \mathbb{C} : \| g\|_{\ell^p(\Z^d)}< \infty \}.
\]
If $p= \infty$ then for $x \in \mathbb{R}^d$ and $f: \Z^d \to \mathbb{C}$ we define
\[
\|x||_{\infty}=\sup_{i \in [d]} |x_i|, \ \ \ \ \|f\|_{\ell^{\infty}(\Z^d)}= \sup_{x \in \Z^d} |f(x)|, \quad \ell^\infty(\Z^d)= \{ g: \Z^d \to \mathbb{C} : \| g\|_{\ell^\infty(\Z^d)}< \infty \}.
\]
 For $f \in \ell^2(\mathbb{Z}^d)$ and $g: \Z^d \to \C$ with finite support we define $f \ast g \in \ell^2(\mathbb{Z}^d)$ by the series
\[f \ast g(x)= \sum_{y \in \mathbb{Z}^d } f(y)g(x-y)=\sum_{y \in \mathbb{Z}^d } f(x-y)g(y).\]
If $f \in \ell^1(\Z ^d)$ we introduce the discrete Fourier transform by the formula
\[
 \widehat{f}(\xi)= \sum_{x \in \Z ^d} f(x) e( x \cdot \xi ), \ \text{ for } \ \xi \in \T^d. 
\]
One can uniquely extend the discrete Fourier transform to $f \in \ell^2(\Z ^d)$ so that $\widehat{f} \in L^2(\T^d)$ and we have the following Parseval identity:
\[
  \|\widehat{f}\|_{L^2(\T^d)}=\| f \|_{\ell^2(\Z^d)}.
\]
Moreover, for any $f \in \ell^2(\Z ^d)$ and $g: \Z^d \to \C$ with finite support , the following holds
\[
 \widehat{f \ast g}(\xi)= \widehat{f}(\xi) \widehat{g}(\xi).
\]
$\mathcal{F}^{-1}$ will denote the inverse of the discrete Fourier transform, that is
\[
\mathcal{F}^{-1}(G)(x)= \int_{\T^d} G(\xi) e(-x \cdot \xi) \ d \xi,
\]
where $G \in L^2(\T ^d)$.
 We let $m^q_t$ be the multiplier symbol \[m^q_t(\xi)=\frac{1}{|B^q_t \cap \mathbb{Z}^d|}  \sum_{x \in B^q_t \cap \mathbb{Z}^d} e(x \cdot \xi). \]
 $\mathbb{T}^d$ will denote the $d$-dimensional torus, which will be identified with the set $[-\frac{1}{2}, \frac{1}{2})^d.$ For $\xi \in \T^d$ we define 
 \[
 \|\xi \|= \Big( \sum_{j=1}^d \sin^2(\pi \xi_j) \Big)^{1/2},
 \]
 note that then $\| \xi+1/2 \|= \Big( \sum_{j=1}^d \cos^2(\pi \xi_j) \Big)^{1/2}.$
 $\Sym(d)$ will denote the permutation group of $\lbrace 1,2,...,d \rbrace$, it naturally acts on $\R^d$ by $\sigma \cdot x=(x_{\sigma(1)},...,x_{\sigma(d)})$.
 Let $Q=[-\frac{1}{2}, \frac{1}{2}]^d$. For convenience we define $\kappa_q(d,N)=\frac{N}{d^{1/q}}$. 
 We will use the convention that $A \lesssim B$ to say that there exists an absolute constant $C>0$ such that $A \leq C B$.
 For $A \subseteq \mathbb{R}^d$ by $|A|$ we will denote Lebesgue measure of $A$ or the number of elements of $A$, this should be clear from the context.
 For $x \in \R^d$ we define 
\[
\supp(x)= \{i \in [d]: x_i \neq 0\}. 
\]
 
\begin{Defn}  \label{def:1.3} The discrete Hardy-Littlewood averaging operator over the $\ell^q$ balls is defined for any function $f:\mathbb{Z}^d \to \mathbb{C}$ by the formula
 \[ \mathcal{M}_t^q f(x)=\frac{1}{|B_t^q \cap \mathbb{Z} ^d |} \sum_{ y \in B_t^q \cap \mathbb{Z} ^d } f(x-y). \]

\end{Defn}
Our goal is to prove the following theorem.
\begin{Thm} \label{thm:1.4}
There exists $C>0$, such that for every $q \geq 1$, $d \in \mathbb{N}$, and $f \in \ell^2(\Z ^d)$ we have \\
\[ \Big\| \sup_{t \in \mathbb{D}, t \leq d^{1/q}} |\mathcal{M}_t^{q} f| \Big\|_{\ell^2(\mathbb{Z}^d)} \leq C \| f \|_{\ell^2(\mathbb{Z}^d)}, \]
where $\mathbb{D}=\lbrace 2^n : n \in \mathbb{N}_0 \rbrace$ is the set of dyadic integers.
\end{Thm}
Using an interpolation argument one can show that Theorem $\ref{thm:1.4}$ implies Theorem $\ref{thm:1.1}$.

 \section{Auxiliary lemmas} In this section we give some preliminary results regarding behavior of points in $B_N^q \cap \Z^d$, which will turn out to be useful in next sections. \par
The argument in the proof of Lemma \ref{lem:2.1} below is very greedy, in the case of $q=2$ it gives worse bound by a factor $2^{d/2}$ compared to \cite[Lemma 5.1]{contANDdiscr}, however it is completely irrelevant for us.
\begin{Lem} \label{lem:2.1}
    Let $d \in \N$, $N>0$, $q \geq 1$. Define $N_1=N+\frac{d^{1/q}}{2}$ then we have
\[ (2 \lfloor \kappa_q(d,N) \rfloor +1)^d \le | B_N^q \cap \mathbb{Z}^d | \leq |B_{N_1}^q|. \]
\end{Lem}
\begin{proof} Lower bound follows from the inclusion $[-\kappa_q(d,N),\kappa_q(d,N)] \subseteq B_N^q$.
\par In terms of the upper bound, notice that for every $x \in B_N^q$ and $z \in Q$ by the triangle inequality we have
\[\| x +z \|_{q} \leq \| x \|_{q} + \|z \|_{q} \leq N+\frac{d^{1/q}}{2}=N_1. \]
From this we obtain
\begin{align*}
| B_N^q \cap \mathbb{Z}^d |&= \hspace{-0.3cm}\sum_{x \in B_N^q \cap \mathbb{Z}^d} \hspace{-0.3cm}1 \leq 	\hspace{-0.3cm}\sum_{x \in B_N^q \cap \mathbb{Z}^d} \int_{Q} \mathds{1}_{ \lbrace z \in Q : \| x+z \|_{q} \leq N_1 \rbrace}(y) dy= \hspace{-0.3cm} \sum_{x \in B_N^q \cap \mathbb{Z}^d} \int_{x+Q} \hspace{-0.3cm} \mathds{1}_{ B_{N_1}^q}(y) dy \leq |B_{N_1}^q|. 
\end{align*}

\end{proof}
\begin{Cor}
    \label{cor:2.2}

    For all $q \geq 1$, $d \in \N$, $N>0$ we have the following bound
\[  | B_N^q \cap \mathbb{Z}^d | \leq 2\Big(\kappa_q(d,N)+\frac{1}{2}\Big)^d \cdot 8^d . \]
\end{Cor}
\begin{proof}
    
Using Lemma \ref{lem:2.1} and the formula for the volume of unit $\ell^q$ ball in $\mathbb{R}^d$ space we obtain
\[| B_N^q \cap \mathbb{Z}^d | \leq |B_{N_1}^q| = \Big( N+\frac{d^{1/q}}{2}\Big)^d \cdot \frac{2^d \Gamma(1+\frac{1}{q})^d}{\Gamma(1+\frac{d}{q})} . \]
Using the following bound for the gamma function 
\[
\sqrt{\frac{2 \pi }{x}} \Big( \frac{x}{e} \Big)^x \leq \Gamma(x),
\]
which holds for all $x>0$ (see for instance \cite[\href{https://dlmf.nist.gov/5.6.E1}{(5.6.1)}]{DLMF})
and the fact that $\Gamma(y) \leq 1 $ for all $y \in [1,2]$, we get
\begin{align*}
&| B_N^q\cap \mathbb{Z}^d | \leq \Big( N+\frac{d^{1/q}}{2}\Big)^d \cdot \frac{2^d \Gamma(1+\frac{1}{q})^d}{\Gamma(1+\frac{d}{q})} \leq \Big(\kappa_q(d,N)+\frac{1}{2} \Big)^d d^{d/q} \cdot \frac{2^d}{\Gamma(1+\frac{d}{q})}
\\
&\leq \Big(\kappa_q(d,N)+\frac{1}{2} \Big)^d d^{d/q} \cdot \frac{2^d e^{1+d/q} \sqrt{(1+d/q)}}{(1+d/q)^{1+d/q}} \frac{1}{\sqrt{2 \pi}} \le \Big(\kappa_q(d,N)+\frac{1}{2} \Big)^d d^{d/q} \cdot \frac{(2e)^d  }{(1+d/q)^{d/q}} \frac{e}{\sqrt{2 \pi}} 
\\
&\le 2 \Big(\kappa_q(d,N)+\frac{1}{2} \Big)^d  \cdot (2e)^d  q^{ \frac{d}{q}} \le 2\Big(\kappa_q(d,N)+\frac{1}{2} \Big)^{d}  \cdot (2e \cdot e^{1/e})^{d} \le 2\Big(\kappa_q(d,N)+\frac{1}{2} \Big)^{d}  \cdot 8^d.
\end{align*}
\end{proof}
The next Lemma is an adapted version of \cite[Lemma 3.2]{balls}, proof goes along the same lines.
\begin{Lem} \label{lem:2.3}
For all $q \geq 1$, $d,N \in \mathbb{N}$, if $\kappa_q(d,N) \leq  409^{-1/q}$ and $N^q \geq k \geq 409\kappa_q(d,N)^q N^q$, then we have
\begin{equation} \label{lem2.3eq1}
 | \lbrace x \in B_N^q \cap \mathbb{Z}^d : |\lbrace i \in [d]:x_i= \pm 1 \rbrace | \leq N^q-k \rbrace| \leq 2^{2-k} |B_N^q \cap \mathbb{Z}^d|.  
\end{equation} 
Moreover for all $q \geq 1$ and $d,N \in \N$ satisfying $\kappa_q(d,N) \le N^{-11}$ we have
\begin{equation} \label{lem2.3eq2}
| \lbrace x \in B_N^q \cap \mathbb{Z}^d : |\lbrace i \in [d]:x_i= \pm 1 \rbrace | \leq N^q-2 \rbrace| \leq  \frac{1}{N^{q}} |B_N^q \cap \mathbb{Z}^d|.  
\end{equation} 
    
\end{Lem}
\begin{proof}
\par We will first prove \eqref{lem2.3eq1}. Let $n=N^q$. We have that
\begin{equation} \label{eq:2.3}
 | \lbrace x \in B_N^q \cap \mathbb{Z}^d : |\lbrace i \in [d] :x_i= \pm 1 \rbrace | \leq n-k \rbrace|= \sum_{m=k}^n |E_m|, 
 \end{equation}
where 
\[ E_m= \lbrace x \in B_N^q \cap \mathbb{Z}^d : |\lbrace i \in [d] :x_i= \pm 1 \rbrace | = n-m \rbrace. \]
To prove \eqref{lem2.3eq1} it is sufficient to show that $|E_m| \leq 2^{1-m} |B_N^q \cap \mathbb{Z}^d|$ holds for every $m \in \lbrace k,k+1,...,n \rbrace$. Notice that if $x \in E_m$ then
\[ | \lbrace i \in [d]: |x_i| \geq 2 \rbrace | \leq \Big\lfloor \frac{m}{2^q} \Big\rfloor. \]
In order to bound $|E_m|$ we will consider two cases separately.
\par \textbf{Case 1) $m< 2^q$.} \\
In this situation every $x \in E_m$ consists only of coordinates $-1,0,1$, hence we have
\begin{equation*}
    |E_m| = 2^{n-m} \binom{d}{n-m}.
\end{equation*}
Moreover 
\begin{equation*} 
    \begin{split}
        &\binom{d}{n-m}  \binom{d}{n}^{-1} \hspace{-0.3cm}= \frac{n!(d-n)!}{(n-m)!(d-n+m)!} \le \Big( \frac{n}{d-n} \Big)^m\hspace{-0.3cm} = \Big( \frac{\kappa_q(d,N)^q}{1-\kappa_q(d,N)^q}  \Big)^m \hspace{-0.3cm}\le \Big( 2\kappa_q(d,N)^q\Big)^m,
    \end{split}
\end{equation*} 
above we have used the fact that $\kappa_q(d,N) \le 2^{-1/q}$.
We also have 
\begin{equation} \label{eq:2.4}
 2^n \binom{d}{n} \leq |B_N^q \cap \mathbb{Z}^d |.
\end{equation}
Combining these three facts we obtain
\begin{align*}
    &|E_m| \le 2^{n-m} \binom{d}{n-m}= 2^{-m} \cdot 2^{n} \binom{d}{n} \binom{d}{n-m}  \binom{d}{n}^{-1} \\
    &\le 2^{-m} |B_N^q \cap \Z^d| \Big( 2\kappa_q(d,N)^q\Big)^m \le 2^{-m} |B_N^q \cap \Z^d|.
\end{align*}
\par \textbf{Case 2) $m \ge 2^q$.} \\
Notice that we have the following upper bound for $|E_m|$ . 
\begin{equation} \label{eq:2.5}
|E_m| \leq 2^{n-m} \binom{d}{n-m} \binom{d-n+m}{\lfloor \frac{m}{2^q} \rfloor} |B_{m^{1/q}}^{(q,\lfloor m/2^q \rfloor)} \cap \mathbb{Z}^{\lfloor m/2^q \rfloor} |, 
\end{equation}
where
\[
B_{m^{1/q}}^{(q,\lfloor m/2^q \rfloor)}= \lbrace y \in \mathbb{R}^{\lfloor m/2^q \rfloor}: \| y \|_{q} \leq m^{1/q} \rbrace.
\]
Indeed, in $\binom{d}{n-m}$ options we choose coordinates on which $x \in E_m$ will have values $\pm 1$, then we choose each sign in 2 ways, this explains the factor $2^{n-m} \binom{d}{n-m}$. Next we choose $\lfloor \frac{m}{2^q} \rfloor$ coordinates  in which the set $\lbrace i \in [d]: |x_i| \geq 2 \rbrace$ will be contained, for that we have $\binom{d-n+m}{\lfloor \frac{m}{2^q} \rfloor}$ options. Lastly we bound the number of ways of putting numbers on these coordinates such that the condition $\| x \|_{q} \leq N$ holds, this is bounded by $|B_{m^{1/q}}^{(q,\lfloor m/2^q \rfloor)} \cap \mathbb{Z}^{\lfloor m/2^q \rfloor}|$. \\ \\
Using Corollary \ref{cor:2.2} we obtain that
\begin{equation} \label{eq:2.6}
    \begin{split}
        |B_{m^{1/q}}^{(q,\lfloor m/2^q \rfloor)} \cap \mathbb{Z}^{\lfloor m/2^q \rfloor} | &\leq 2 \cdot \Big( \frac{m^{1/q}}{\lfloor m/2^q \rfloor^{1/q}}+\frac{1}{2} \Big)^{\lfloor m/2^q \rfloor} 8^{\lfloor m/2^q \rfloor} \leq 2 \cdot \Big( \Big( \frac{m}{m/2^{q+1}} \Big)^{1/q}+ \frac{1}{2} \Big)^{m/2^q} 8^{\lfloor m/2^q \rfloor}  \\
 &= 2 \cdot\Big( 2^{1+1/q}+ \frac{1}{2} \Big)^{m/2^q} 8^{\lfloor m/2^q \rfloor} \leq  2 \cdot\Big( \frac{9}{2} \Big)^{m/2} 8^{m/2}= 2 \cdot 6^{m},
    \end{split}
\end{equation}
above we have used the fact that $m  \geq 2^{q}$.
We also have that
\begin{equation} \label{eq:2.7}
    \begin{split}
        &\binom{d}{n-m} \binom{d-n+m}{\lfloor \frac{m}{2^q} \rfloor} \binom{d}{n}^{-1} = \frac{n! \cdot (d-n)!}{(n-m)! \lfloor m/2^q \rfloor! \cdot (d-n+m-\lfloor m/2^q \rfloor)!}  \\
        &\leq \frac{n^{\lfloor m/2^q \rfloor}}{\lfloor m/2^q \rfloor!} \Big( \frac{n}{d-n} \Big)^{m-\lfloor m/2^q \rfloor} \stackrel{(*)}{\leq} \Big( \frac{en}{\lfloor m/2^q \rfloor} \Big)^{\lfloor m/2^q \rfloor} \Big( \frac{n}{d-n} \Big)^{m-\lfloor m/2^q \rfloor}  \\ 
        &\stackrel{(**)}{\leq} \Big( \frac{2^qen}{m} \Big)^{ m/2^q } \Big( \frac{n}{d-n} \Big)^{m-\lfloor m/2^q \rfloor}= \Big( \frac{2^qen}{m} \Big)^{ m/2^q } \Big( \frac{\kappa_q(d,N)^q}{1-\kappa_q(d,N)^q} \Big)^{m-\lfloor m/2^q \rfloor}  \\
    &\leq \Big( \frac{2^qen}{m} \Big)^{ m/2^q } \Big( 2\kappa_q(d,N)^q \Big)^{m-\lfloor m/2^q \rfloor} \leq \Big( \frac{2^q en}{m} \Big)^{ m/2^q } \Big( 2\kappa_q(d,N)^q \Big)^{m/2}.
    \end{split}
\end{equation} 
$(*)$ holds, since $\lfloor m/2^q \rfloor! \geq (\lfloor m/2^q \rfloor/e)^{\lfloor m/2^q \rfloor}$. In $(**)$ we used the fact that for any $a>0$ the function $(0,a/e] \ni t \mapsto (a/t)^t$ is increasing. Two last inequalities hold, since $m-\lfloor m/2^q \rfloor \geq m-m/2^q \geq m/2$ and $\kappa_q(d,N) \leq 2^{-1/q}$. \\
Using these four inequalities (i.e. \eqref{eq:2.4}, \eqref{eq:2.5}, \eqref{eq:2.6}, \eqref{eq:2.7}) and our assumptions on $\kappa_q(d,N)$ we finally obtain
\begin{equation} \label{eq:2.8}
\begin{split}    
|E_m| &\leq 2 \cdot2^n \binom{d}{n} 2^{-m}  \Big( \frac{2^qen}{m} \Big)^{ m/2^q } \Big( 2\kappa_q(d,N)^q \Big)^{m/2} 6^m \\
 &\leq  2^{1-m}  \cdot |B_N^q \cap \mathbb{Z}^d| \Big(\frac{n\kappa_q(d,N)^q}{m}\Big)^{m/2} e^{m/e}e^{m/2} 2^{m/2}  6^m\\
 &\leq  2^{1-m} |B_N^q \cap \mathbb{Z}^d| \Big(\frac{n\kappa_q(d,N)^q}{m}\Big)^{m/2} 20.211^m \leq  2^{1-m} |B_N^q \cap \mathbb{Z}^d|.
 \end{split}
 \end{equation}
Last inequality holds by the assumption $m \geq k \geq 409n \kappa_q(d,N)^q $. Due to \eqref{eq:2.3} this concludes the proof of \eqref{lem2.3eq1}.
\par Now we prove the second part, that is \eqref{lem2.3eq2}. Assume that $d, N \in \mathbb{N}$ satisfy $\kappa_q(d,N) \le N^{-11}$. We will also assume that $n \geq 2$, since in the case $N^q=n < 2$ the statement trivially holds, because then the left-hand side of \eqref{lem2.3eq2} equals $0$. 
We proceed as in the first part of the proof with minor changes. If $2 \le m<2^q$, then we have
\begin{align*}
  |E_m| \le \Big( 2\kappa_q(d,N)^q\Big)^m 2^{-m} |B_N^q \cap \Z^d| \le \Big( \frac{2}{n^{11}}\Big)^m   2^{-m} |B_N^q \cap \Z^d| \le \frac{1}{n} 2^{-m} |B_N^q \cap \Z^d|.
\end{align*}
On the other hand if $m \ge 2^q$, then by the third inequality of \eqref{eq:2.8}
we have
\begin{align*}
&|E_m| \le 2^{1-m} |B_N^q \cap \mathbb{Z}^d|  \Big(\frac{n\kappa_q(d,N)^q}{m}\Big)^{m/2} 20.211^m \le  2^{1-m} |B_N^q \cap \Z^d| \Big(\frac{n^{-10} 409}{m}\Big)^{m/2} \\
&\le 2^{1-m} |B_N^q \cap \Z^d| \Big( \frac{n^{-1}409}{2^{9}m} \Big)^{m/2}
\hspace{-0.3cm}\le  2^{1-m} |B_N^q \cap \Z^d| \Big( \frac{1}{n} \Big)^{m/2} \hspace{-0.3cm}\leq  \frac{2}{n} 2^{-m} |B_N^q \cap \Z^d|,
\end{align*}
 by \eqref{eq:2.3} with $k=2$ we get \eqref{lem2.3eq2}.  
\end{proof}
In order to prove Lemma \ref{sum_z} we need to introduce an auxiliary lemma stating that in certain range it's unlikely for points in $B_N^q \cap \Z^d$ to have a lot of big coordinates.
\begin{Lem} \label{lem:2.4}
Take any $q \geq 1$, assume that $d,N \in \N$ satisfy $\kappa_q(d,N) \in [N^{-11}, e^{-12/q}] $. Then for every $k \in \N, k>1$ we have 
\[
\Big| \Big\{ x \in B_N^q \cap \Z^d: |\{ i \in [d]: |x_i|=k \}| \geq \frac{\kappa_q(d,N)^{q/22}N^q}{k^{6}} \Big\}\Big| \lesssim  \frac{1}{N^{2q}} |B_N^q \cap \Z^d|.
\]

\end{Lem}
Proof of this lemma is just based on removing coordinates $\pm k$ and increasing number of $\pm 1$ coordinates.
\begin{proof} Let $n=N^q$.
Fix $k \in \N, k>1$. Every $x \in \Z^d$ can be written as $x=u(x)+v(x)$, where
\[
u(x)_j= \begin{cases}
    x_j, \ \ \text{if $|x_j| \notin \{1,k \} $}, \\
    0, \ \ \text{otherwise}.
\end{cases}
\]
Let $U(B_N^q \cap \Z^d)= \big\{u(x): x \in B_N^q \cap \Z^d \big\}$. Then we have the following disjoint decomposition
\begin{equation} \label{eq:2.9}
    B_N^q \cap \Z^d= \bigcup_{u \in U(B_N \cap \Z^d)} \bigcup_{\substack{0 \le l,m \le n
    }} A(u,l,m),
\end{equation}
where
\[
A(u,l,m)=\Big\{ x \in B_N^q \cap \Z^d: u(x)=u, |\{i \in [d]: |x_i|=k\}|=l, |\{i \in [d]: |x_i|=1\}|=m \Big\}.
\]
We will prove the following. If $l \geq \frac{\kappa_q(d,N)^{q/22}n}{k^6}$, then for all $u$ and $m$ we have
\begin{equation} \label{eq:2.10}
|A(u,l,m)| \lesssim \frac{1}{n^3} |A(u,0,m+\lfloor l k^q \rfloor)|.
\end{equation}
Note that the above suffices to prove Lemma \ref{lem:2.4}, indeed, we have
\begin{align*}
& \Big|\Big\{ x \in B_N^q \cap \Z^d: |\{ i \in [d]: |x_i|=k \}| \geq \frac{\kappa_q(d,N)^{q/22}N^q}{k^6} \Big\}\Big|  =\hspace{-0.5cm} \sum_{u \in U(B_N^q \cap \Z^d)} \sum_{0 \le m \le n} \sum_{\frac{\kappa_q(d,N)^{q/22}n}{k^6} \le l \le n}\hspace{-0.7cm} |A(u,l,m)| \\
&\lesssim  \frac{1}{n^3}\hspace{-0.3cm} \sum_{u \in U(B_N^q \cap \Z^d)} \sum_{0 \le m \le n} \sum_{\frac{\kappa_q(d,N)^{q/22}n}{k^6} \le l \le n}\hspace{-0.7cm} |A(u,0,m+\lfloor l k^q \rfloor)| 
\lesssim \frac{1}{n^2}\hspace{-0.3cm} \sum_{u \in U(B_N^q \cap \Z^d)} \sum_{0 \le t \le n} |A(u,0,t)| \lesssim \hspace{-0.1cm}\frac{1}{n^2} |B_N^q \cap \Z^d|,
\end{align*}
penultimate inequality follows from the fact, that for each $t \in [0,n] \cap \N_0$ there are at most $n+1$ pairs of numbers $m,l \in [0,n] \cap \N_0$ such that $t=m+\lfloor l k^q \rfloor$. \par Now we will prove \eqref{eq:2.10}, take any $u \in U(B_N^q \cap \Z^d)$, $l,m \in [0,n] \cap \N_0$. We can assume that 
\[
m+lk^q + \|u \|_q^q \le n,
\]
otherwise $A(u,l,m)= \emptyset$. Let $|u|=|\supp(u)| \le n $. Then we have
\begin{equation} \label{eq:2.11}
    \begin{split}
&\frac{|A(u,l,m)|}{|A(u,0,m+\lfloor l k^q \rfloor)|}= 2^{m+l} \binom{d-|u|}{l} \binom{d-|u|-l}{m} \cdot \Bigg( 2^{m+\lfloor l k^q \rfloor}   \binom{d-|u|}{m+\lfloor l k^q \rfloor} \Bigg)^{-1} \\
&= 2^{l-\lfloor lk^q \rfloor} \frac{(d-|u|)!}{l! (d-|u|-l)!} \cdot \frac{(d-|u|-l)!}{m!(d-|u|-l-m)!} \frac{(m+\lfloor l k^q \rfloor)! (d-|u|-m-\lfloor l k^q \rfloor)!}{(d-|u|)!} \\
&\le 2^{l-\lfloor lk^q \rfloor} \frac{(m+\lfloor l k^q \rfloor)!}{m! l!} \cdot \frac{(d-|u|-m-\lfloor l k^q \rfloor)!}{(d-|u|-l-m)!} \\
&\le 2^{l-\lfloor lk^q \rfloor} \frac{(m+\lfloor l k^q \rfloor)^{\lfloor l k^q \rfloor} e^{l}}{l^l} (d-|u|-m-\lfloor l k^q \rfloor)^{-\lfloor l k^q \rfloor+l} \leq n^{\lfloor l k^q \rfloor} e^l l^{-l} d^{-\lfloor l k^q \rfloor+l}.
\end{split}
\end{equation}
Above we have used the fact that $l! \ge l^l e^{-l}$ and
\[
|u|+m+\lfloor l k^q \rfloor \le \|u\|_q^q+m+\lfloor l k^q \rfloor \le n \le d/2.
\]
Further simplifying right-hand side of last inequality of \eqref{eq:2.11} and incorporating assumptions on $l$ and $k$ we obtain
\begin{equation} \label{eq:2.12}
    \begin{split}
        &\frac{|A(u,l,m)|}{|A(u,0,m+\lfloor l k^q \rfloor)|}\le n^{\lfloor l k^q \rfloor} e^l l^{-l} d^{-\lfloor l k^q \rfloor+l}= \Big( \kappa_q(d,N)^q \Big)^{\lfloor l k^q \rfloor} \Big( \frac{e d}{l} \Big)^{l} \\
& \le \Big( \kappa_q(d,N)^q \Big)^{\lfloor l k^q \rfloor}  \Big(e k^6 \kappa_q(d,N)^{-23q/22} \Big)^{l} \le \Big( \kappa_q(d,N)^q \Big)^{lk} \Big( \kappa_q(d,N)^{q} \Big)^{-23l/22}   k^{6l} e^l\\
&\le \kappa_q(d,N)^{qlk(1-23/44)} \cdot 2^{6lk}  \cdot e^l\le e^{
kl(6 \log2+1/2-\frac{12\cdot21}{44})
}\le e^{-lk}.
    \end{split}
\end{equation}
Last inequality follows from the assumption $\kappa_q(d,N) \le e^{-12/q}$. In view of \eqref{eq:2.12}, in order to proof \eqref{eq:2.10} and hence Lemma \ref{lem:2.4} , it suffices to show that for all $k \geq 2$ and $l \in \N_0, l \ge \frac{\kappa_q(d,N)^{q/22}n}{k^6}$ we have
\begin{equation} \label{eq:2.13}
    e^{-lk} \lesssim \frac{1}{n^3}.
\end{equation}
We will consider two cases. \\
\textbf{If $k \le n^{1/24}$.} Then by $\kappa_q(d,N)^q \ge n^{-11}$ we get
\[
l \ge \frac{\kappa_q(d,N)^{q/22}n}{k^6} \ge \frac{n^{1/2}}{k^6} \ge n^{1/4},
\]
hence
$e^{-kl} \le e^{-2n^{1/4}} \lesssim\frac{1}{n^3}$,
thus \eqref{eq:2.13} holds. \\
\textbf{If $k \ge n^{1/24}$.} Then assumptions $l \in \N_0$ and $l \ge \frac{\kappa_q(d,N)^{q/22}n}{k^{6}}>0$ imply $l \geq 1$. From this we obtain
\[
e^{-kl} \le e^{-n^{1/24}} \lesssim \frac{1}{n^3}.
\]
Thus in any case \eqref{eq:2.13} holds, this concludes the proof of \eqref{eq:2.10} and hence concludes the proof of Lemma \ref{lem:2.4}.
\end{proof}
Now we can introduce a lemma, which will come in handy during the proof of Proposition \ref{prop:3.5}. 
\begin{Lem}
    \label{sum_z} For every $q \geq 1$ and $d,N \in \N$ satisfying $\kappa_q(d,N) \leq e^{-12/q}$ we have
    \[
    \Big|\Big\{ x \in B_N^q \cap \Z^d: \sum_{\substack{i=1 \\|x_i| \geq 2}}^d x_i^4 >\kappa_q(d,N)^{q/22} N^q \Big\}\Big| \lesssim \frac{1}{N^q} |B_N^q \cap \Z^d|.
    \]
\end{Lem}
\begin{proof}  We will consider two cases \\
\textbf{Case 1) If $\kappa_q(d,N) \le N^{-11}$.} Then by Lemma \ref{lem:2.3} we get 
\begin{align*}
    \Big|\Big\{ x \in B_N^q \cap \Z^d: \sum_{\substack{i=1 \\|x_i| \geq 2}}^d x_i^4 >0\Big\}\Big| \le | \lbrace x \in B_N^q \cap \mathbb{Z}^d : |\lbrace i \in [d]:x_i= \pm 1 \rbrace | \leq N^q-2 \rbrace| \le \frac{1}{N^q} |B_N^q \cap \mathbb{Z}^d|,
\end{align*}
so the claim of Lemma \ref{sum_z} holds. 
\\
\textbf{Case 2) If $\kappa_q(d,N) \in [N^{-11}, e^{-12/q}]$.} In this case we can apply Lemma \ref{lem:2.4}. Let 
\[
E=\Big\{ x \in B_N^q \cap \Z^d: (\exists k )  |\{ i \in [d]:  |x_i|=k \} | \ge \frac{\kappa_q(d,N)^{q/22} N^q}{k^{6}}\Big\}.
\]
Then by Lemma \ref{lem:2.4} we have
\begin{equation} \label{eq:2.14}
    |E| \le \sum_{2 \le k \le N^q} \Big| \Big\{ x \in B_N^q \cap \Z^d:  |\{ i \in [d]:  |x_i|=k \} | \ge \frac{\kappa_q(d,N)^{q/22} N^q}{k^{6}}\Big\} \Big| \lesssim \frac{1}{N^q} |B_N \cap \Z^d|.
\end{equation}
Moreover, notice that for all $x \in B_N^q \cap \Z^d \setminus E$ we have 
\begin{equation*}
 \sum_{\substack{i=1, \\ |x_i| \ge 2}}^d x_i^4= \sum_{k \geq 2} k^4 |\{ i \in [d]:  |x_i|=k \} | \le \sum_{k \ge 2} k^4 \cdot \frac{\kappa_q(d,N)^{q/22} N^q}{k^{6}}<\kappa_q(d,N)^{q/22} N^q.
\end{equation*}
The above combined with \eqref{eq:2.14} gives
\begin{equation*}
     \Big|\Big\{ x \in B_N^q \cap \Z^d: \sum_{\substack{i=1 \\|x_i| \geq 2}}^d x_i^4 > \kappa_q(d,N)^{q/22} N^q \Big\}\Big| \le |E| \lesssim \frac{1}{N^q} |B_N^q \cap \Z^d|.
\end{equation*}
Thus in any case conclusion of Lemma \ref{sum_z} holds.
\end{proof}
Lemma \ref{lem:2.6} will be useful to combine various bounds coming from item \eqref{p3.4(2)} of Proposition \ref{prop:3.4}.
\begin{Lem}[{\cite[Lemma 2.6]{balls}}]
\label{lem:2.6} Assume that we have a sequence $(u_j: j \in [d])$ with $0 \leq u_j \leq \frac{1- \delta_0}{2}$ for some $\delta_0 \in (0,1)$. Suppose that $I \subseteq [d]$ satisfies $\delta_1d \leq |I| \leq d$ for some $ \delta_1 \in (0,1]$. Then for every $J=(d_0,d] \cap \Z$ with $0 \leq d_0 \leq d$ we have
\[
\frac{1}{d!}\sum_{\tau \in \Sym(d)} \exp\Big(- \sum_{j \in \tau(I) \cap J} u_j\Big) \leq 3 \exp\Big(-\frac{\delta_0 \delta_1}{20} \sum_{j \in J} u_j \Big).
\]
\end{Lem}
\noindent In \cite[Lemma 2.6]{balls} there is an extra assumption, that sequence $(u_j: j \in [d])$ is decreasing, however this assumption can be removed by rearranging the sequence, see \cite[Lemma 6.7]{MiSzWr}. Assumption regarding upper bound on $u_j$ seems to be arbitrary, for this reason we remove it.
\begin{Cor}
   \label{Cor:2.7}Let $M \in \N$. Assume that we have a sequence $(u_j: j \in [d])$ with $0 \leq u_j \leq M \cdot \frac{1- \delta_0}{2}$ for some $\delta_0 \in (0,1)$. Suppose that $I \subseteq [d]$ satisfies $\delta_1d \leq |I| \leq d$ for some $ \delta_1 \in (0,1]$. Then for every $J=(d_0,d] \cap \Z$ with $0 \leq d_0 \leq d$ we have
\[
\frac{1}{d!}\sum_{\tau \in \Sym(d)} \exp\Big(- \sum_{j \in \tau(I) \cap J} u_j\Big) \leq 3 \exp\Big(-\frac{\delta_0 \delta_1}{20M} \sum_{j \in J} u_j \Big).
\] 
\end{Cor}
\begin{proof}
 \[
\frac{1}{d!}\sum_{\tau \in \Sym(d)} \exp\Big(- \sum_{j \in \tau(I) \cap J} u_j\Big) \leq \frac{1}{d!}\sum_{\tau \in \Sym(d)} \exp\Big(- \sum_{j \in \tau(I) \cap J} \frac{u_j}{M}\Big) \le 3 \exp\Big(-\frac{\delta_0 \delta_1}{20M} \sum_{j \in J} u_j \Big).
\]    
\end{proof}

\section{Replacement by an auxiliary multiplier symbol}
In this section using we first define $\beta_n^J$ which is simpler multiplier than $m_N^q$ defined in section \ref{sec1.1}. Then we will obtain crucial bounds on $\beta_n^J$, which combined with lemmas of previous section will allow us to bound $m_N^q$. The ideas of this section are contained in the proof of \cite[Proposition 3.3]{balls}. We give slightly refined approach, which appeared in very big generality in \cite[Sections 2,4]{NW}.
\begin{Defn}
    For every $n,d \in \N$, $J\subseteq [d]$ with $n \leq |J|$ we define $\beta_n^J: \T^d \to \mathbb{C}$ by the formula
    \[
    \beta_{n}^J(\xi)=  \frac{1}{|D_n^J|} \sum_{x \in D_n^J}e( x \cdot \xi),
    \]
    where
    \[
    D_n^J=\Big\{ y \in \{-1,0,1\}^d: |\{i \in [d]: |y_i|=1\}|=n, \supp(y) \subseteq J \Big\}.
    \]
\end{Defn}
In order to better understand $\beta_n^J(\xi)$ it is useful to recall definition and basic properties of Krawtchouk polynomials.
\begin{Defn}
(Krawtchouk polynomial). For every $n \in \N _0$,
$k \in \lbrace 0,1,...,n \rbrace$ and $x \in \mathbb{R}$ we define $k$-th Krawtchouk polynomial by the formula
\[ \kr _k^{(n)}(x)= \frac{1}{\binom{n}{k}} \sum_{j=0}^k (-1)^j \binom{x}{j} \binom{n-x}{k-j}, 
\]
if $x$ is not integer, then we use Pochhammer symbol  $\binom{x}{j}=\frac{x(x-1)...(x-j+1)}{j!}$. \label{def:3.2}
\end{Defn}
Next theorem describes important facts regarding Krawtchouk polynomials.
\begin{Thm} \label{thm:3.3}
For every $n \in \N _0$ and integers $x,k \in [0,n]$ we have
\begin{enumerate}
    \item{Symmetry:} $\kr _k^{(n)}(x)= \kr_x^{(n)}(k)$.
    \item{Reflection symmetry:} $\kr ^{(n)}_k(n-x)=(-1)^k \kr _k^{(n)}(x)  $.
    \item{Uniform bound:} \label{thm:3.3.5}There is a constant $c \in (0,1)$ (independent of $n$) such that for any $x,k \leq n/2$ we have
    \[|\kr ^{(n)}_k(x)| \leq e^{-ckx/n}.\]
\end{enumerate} 
\end{Thm}
Proof of the first two points is simple and contained in \cite{412678}, the last point is \cite[Lemma 2.2]{v010a003}.
Now we can state two crucial bounds for the multiplier symbol $\beta_n^J$. From now on $c$ will denote constant from the theorem above.
\begin{Prop} \label{prop:3.4}
 For every $n,d \in \N$, $J\subseteq [d]$ with $n \leq |J|/2$ and $\xi \in \mathbb{T}^d$ we have
     \[ \tag{1} \label{p3.4(1)}
     |\beta_n^J(\xi)-1| \leq  2\frac{n}{|J|} \sum_{i \in J} \sin^2(\pi \xi_i),
     \]
     \[ \tag{2} \label{p3.4(2)}
     |\beta_n^J(\xi)| \leq 2 e^{- \frac{cn}{2|J|} \min\big( \sum_{i \in J} \sin^2(\pi \xi_i), \sum_{i \in J} \cos^2(\pi \xi_i)  \big)}.
     \]
\end{Prop}
\begin{proof}
    We will prove item \eqref{p3.4(1)} first. 
Notice that for any $\epsilon \in \lbrace -1,1 \rbrace^d$ we have 
\begin{align*} 
x \in D_n^J
\iff \epsilon  x \in D_n^J,
\end{align*}
where $\epsilon x=(\epsilon_1 x_1,..., \epsilon_d x_d)$.
This implies that
\[\beta_n^J(\xi)= \frac{1}{|D_n^J|} \sum_{ \epsilon  x \in D_n^J} e(\epsilon  x \cdot \xi)= \frac{1}{|D_n^J|} \sum_{ x \in D_n^J} e( \epsilon  x \cdot \xi). \]
Hence
\begin{align*}
\beta_n^J(\xi)&= \frac{1}{2^d} \sum_{\epsilon \in \lbrace -1,1 \rbrace^d} \beta_n^J(\xi)= \frac{1}{2^d} \sum_{\epsilon \in \lbrace -1,1 \rbrace^d} \frac{1}{|D_n^J|} \sum_{ x \in D_n^J} e( \epsilon  x \cdot \xi)
\\
&=\frac{1}{|D_n^J|} \sum_{ x \in D_n^J} \frac{1}{2^d} \sum_{\epsilon \in \lbrace -1,1 \rbrace^d} e( \epsilon  x \cdot \xi)= \frac{1}{|D_n^J|} \sum_{ x \in D_n^J} \prod_{j=1}^d \cos(2 \pi x_j \xi_j).
\end{align*}
Note that for any sequences of complex numbers $\lbrace a_j \rbrace_{j=1}^d, \lbrace b_j \rbrace_{j=1}^d $ such that $\max_{1 \leq j \leq d} |a_j| \leq 1$ and $\max_{1 \leq j \leq d} |b_j| \leq 1$ we have
\begin{equation}   
\Big| \prod_{j=1}^d a_j - \prod_{j=1}^d b_j \Big| \leq \sum_{j=1}^d |a_j-b_j|, \label{eq:3.1}
\end{equation}

this follows from the formula $\prod_{j=1}^d a_j-\prod_{j=1}^db_j= \sum_{j=1}^d\Big((a_j-b_j)\prod_{l=1}^{j-1}a_l \prod_{k=j+1}^db_k \Big).$ 
Using \eqref{eq:3.1} and the formula $\cos(2x)=1-2\sin^2(x)$ we obtain
\begin{align*}
&|\beta_n^J(\xi)-1|  \leq \frac{1}{|D_n^J|} \sum_{ x \in D_n^J} \Big| \prod_{j=1}^d \cos(2 \pi x_j \xi_j)-1 \Big| \leq \frac{1}{|D_n^J|} \sum_{ x \in D_n^J} \sum_{j=1}^d |\cos(2 \pi x_j \xi_j)-1| \\ 
&= \frac{2}{|D_n^J|} \sum_{ x \in D_n} \sum_{j=1}^d \sin^2(\pi x_j \xi_j) = \frac{2}{|D_n^J|} \sum_{x \in D_n^J} \sum_{j \in \supp(x) } \sin^2(\pi\xi_j)  \\
&= \frac{2}{|D_n^J|} \sum_{j \in J} \sin^2(\pi\xi_j)  | \{ x \in D_n^J: |x_j|=1 \}|
=2 \sum_{j \in J}\sin^2(\pi\xi_j) \cdot \frac{2^{n} \binom{|J|-1}{n-1}}{2^n \binom{|J|}{n}}= 2\frac{n}{|J|}\sum_{j \in J} \sin^2(2 \pi \xi_j).
\end{align*}
This finishes the proof of point \eqref{p3.4(1)} of the proposition. 
\par Now we will prove point \eqref{p3.4(2)} of the proposition. We have
\begin{align*}
&\beta_n^J(\xi)= \frac{1}{|D_n^J|} \sum_{x \in D_n^J} \prod_{j=1}^d \cos(2 \pi x_j \xi_j) \\
&= \frac{1}{2^n \binom{|J|}{n}} \sum_{\substack{I \subseteq J, \\ |I|=n}} \prod_{i \in I} \cos(2 \pi \xi_i) \cdot | \{ x \in D_{n}^J: \supp(x)=I \}|
= \frac{1}{\binom{|J|}{n}} \sum_{\substack{I \subseteq J, \\ |I|=n}} \prod_{i \in I} \cos(2 \pi \xi_i).
\end{align*}
Fix $I \subseteq J$, let $\epsilon(I) \in \{-1,1\}^J$ be such that $\epsilon(I)_i=-1$ exactly when $i \in I$, then we have
\[
\prod_{i \in I} \cos(2 \pi \xi_i)= \prod_{i \in J}\bigg( \frac{1+\cos(2 \pi \xi_i)}{2} + \epsilon(I)_i \frac{1-\cos(2\pi \xi_i)}{2} \bigg) \]\[
= \prod_{i \in J} \big( \cos^2( \pi \xi_i)+ \epsilon(I)_i\sin^2(\pi \xi_i) \big)
= \sum_{U \subseteq J}w_U(\epsilon(I)) \prod_{i \in J \setminus U} \cos^2( \pi \xi_i) \prod_{i \in U} \sin^2(\pi \xi_i),
\]
where we have defined $w_U:\{-1,1\}^J\to \{-1,1 \}$ by the formula $w_U(\epsilon)= \prod_{i \in U} \epsilon_i$. Notice that for any $U \subseteq J$ the following holds
\begin{align*}
&\frac{1}{\binom{|J|}{n}} \sum_{\substack{I \subseteq J, \\ |I|=n}} w_U(\epsilon(I))= \frac{1}{\binom{|J|}{n}} \sum_{\substack{I \subseteq J, \\ |I|=n}} (-1)^{|U \cap I|} = \frac{1}{\binom{|J|}{n}} \sum_{m=0}^n (-1)^m \cdot |\{I \subseteq J: |I|=n, |U \cap I|=m \}| \\
&=\frac{1}{\binom{|J|}{n}} \sum_{m=0}^n (-1)^m \binom{|U|}{m} \binom{|J|-|U|}{n-m}=\kr_n^{(|J|)}(|U|).
\end{align*}
From the above we get that 
\begin{equation} \label{eq:3.2}
\begin{split}
    &\beta_n^J(\xi)=\frac{1}{\binom{|J|}{n}} \sum_{\substack{I \subseteq J, \\ |I|=n}} \prod_{i \in I} \cos(2 \pi \xi_i)= \frac{1}{\binom{|J|}{n}} \sum_{\substack{I \subseteq J, \\ |I|=n}} \sum_{U \subseteq J}w_U(\epsilon(I)) \prod_{i \in J \setminus U} \cos^2( \pi \xi_i) \prod_{i \in U} \sin^2(\pi \xi_i) \\
    &=\hspace{-0.2cm}\sum_{U \subseteq J} \prod_{i \in J \setminus U} \cos^2( \pi \xi_i) \prod_{i \in U} \sin^2(\pi \xi_i)  \frac{1}{\binom{|J|}{n}} \sum_{\substack{I \subseteq J, \\ |I|=n}} w_U(\epsilon(I)) =\hspace{-0.2cm}\sum_{U \subseteq J} \prod_{i \in J \setminus U} \cos^2( \pi \xi_i) \prod_{i \in U} \sin^2(\pi \xi_i) \kr_n^{(|J|)}(|U|).
\end{split}
\end{equation}
Now we will use Theorem \ref{thm:3.3} to bound $\kr_n^{(|J|)}(|U|)$ for every $U \subseteq J$. If $|U| \leq |J|/2$, then by the last point of Theorem \ref{thm:3.3} we have
\[
|\kr_{n}^{(|J|)}(|U|)| \le e^{-\frac{cn|U|}{|J|}},
\]
on the other hand if $|U|>|J|/2$, then using last two points of Theorem \ref{thm:3.3} we get
\[
|\kr_{n}^{(|J|)}(|U|)|=|\kr_{n}^{(|J|)}(|J|-|U|)| \le e^{-\frac{cn(|J|-|U|)}{|J|}}.
\]
In either case we obtain
\[
|\kr_{n}^{(|J|)}(|U|)| \le e^{-\frac{cn|U|}{|J|}}+e^{-\frac{cn(|J|-|U|)}{|J|}}.
\]
Plugging the above into \eqref{eq:3.2} we get
\begin{equation} \label{eq:3.3}
\begin{split}
&|\beta_n^J(\xi)| \le \sum_{U \subseteq J} \prod_{i \in J \setminus U} \cos^2( \pi \xi_i) \prod_{i \in U} \sin^2(\pi \xi_i) |\kr_n^d(|U|)| \\
&\le \sum_{U \subseteq J} \Big( \prod_{i \in J \setminus U} \cos^2( \pi \xi_i) \prod_{i \in U} \sin^2(\pi \xi_i) \Big) e^{-\frac{cn|U|}{|J|}}+  \sum_{U \subseteq J} \Big( \prod_{i \in J \setminus U} \cos^2( \pi \xi_i) \prod_{i \in U} \sin^2(\pi \xi_i) \Big) e^{-\frac{cn(|J|-|U|)}{|J|}} \\
&= \sum_{U \subseteq J}  \prod_{i \in J \setminus U} \cos^2( \pi \xi_i) \prod_{i \in U} \big( \sin^2(\pi \xi_i) e^{-\frac{cn}{|J|}} \big)+\sum_{U \subseteq J}  \prod_{i \in J \setminus U}  \big(\cos^2( \pi \xi_i) e^{-\frac{cn}{|J|}} \big) \prod_{i \in U}  \sin^2(\pi \xi_i):=S_1+S_2.
\end{split}
\end{equation}
For $S_1$ we have
\begin{align*}
    S_1&= \prod_{i\in J}\big( \cos^2( \pi \xi_i)+ e^{-\frac{cn}{|J|}} \sin^2( \pi \xi_i) \big) = \prod_{i \in J} \big(1-(1-e^{-\frac{cn}{|J|}}) \sin^2( \pi \xi_i) \big) \\
    &\le \prod_{i \in J} \exp\big(-(1-e^{-\frac{cn}{|J|}}) \sin^2( \pi \xi_i) \big) \le e^{ - \frac{cn}{2|J|} \sum_{i \in J} \sin^2( \pi \xi_i) },
\end{align*}
above we have used two elementary inequalities $1-x \le e^{-x}$ and $x e^{-x/2} \le 1-e^{-x}$ for $x\geq 0$.
Similarly one can prove that
\begin{equation*}
    S_2 \le e^{- \frac{cn}{2|J|} \sum_{i \in J} \cos^2( \pi \xi_i)},
\end{equation*}
combining these bounds on $S_1,S_2$ with \eqref{eq:3.3} we get the desired conclusion.
\end{proof}
We conclude this section with crucial bounds on the multiplier $m_N^q(\xi)$, which will be deduced from Proposition \ref{prop:3.4} and lemmas of section 2.
\begin{Prop} \label{prop:3.5}
Let $q \geq 1$, $d,N \in \N$ satisfy $\kappa_q(d,N) \le e^{-12/q}$. Then for every $\xi \in \T^d$ we have \\
    \item 
    \begin{equation} \label{eq:3.4}
        |m_N^q(\xi)| \lesssim e^{- \frac{c\kappa_q(d,N)^q}{320} \min\big(\|\xi\|^2, \| \xi +1/2\|^2 \big)} +\frac{1}{N^q}.
    \end{equation}
     If $\|\xi \| \le \| \xi+1/2\|$, then
    \begin{equation} \label{eq:3.5}
    |m_N^q(\xi)-1| \lesssim \kappa_q(d,N)^q \| \xi\|^2+\frac{1}{N^{q}}+ \kappa_q(d,N)^{q/50}.
    \end{equation}
     If $\|\xi \| > \| \xi+1/2\|$, then
    \begin{equation} \label{eq:3.6}
    |m_N^q(\xi)-\frac{1}{|B_N^q \cap \Z^d|} \sum_{x \in B_N^q \cap \Z^d} (-1)^{\sum_{i=1}^d x_i}| \lesssim  \kappa_q(d,N)^q \|\xi+1/2\|^2+ \frac{1}{N^{q}}+ \kappa_q(d,N)^{q/50}.
    \end{equation}
\end{Prop}
It is worth pointing out, that inequalities \eqref{eq:3.5}, \eqref{eq:3.6}  in the case $q \ge 2$ have way simpler proofs relying on H\"older's inequality, see for instance \cite[Proposition 4.1]{Bq}.
\begin{proof}
    Let $n=N^q$. We start with proof of \eqref{eq:3.4}. Let
    \begin{align*}
        E&= \Big\{x \in B_N^q \cap \Z^d: |\{ i \in [d]: |x_i|=1 \}|>n/2, \sum_{\substack{i=1 \\|x_i| \geq 2}}^d x_i^4 \le \kappa_q(d,N)^{q/22} n \Big\}, 
    \end{align*}
     by Lemmas \ref{lem:2.3} (with $k=\lfloor n/2 \rfloor$) and \ref{sum_z} we have
    \begin{align} \label{eq:3.7}
    |B_N^q \cap \Z^d \setminus E| &\lesssim \Big(\frac{1}{n}+2^{-n/2}\Big) |B_N^q \cap \Z^d| \lesssim \frac{1}{n} |B_N^q \cap \Z^d|.
    \end{align}
    We define
    \begin{align*}
        s_N^q(\xi)= \frac{1}{|B_N^q \cap \Z^d|} \sum_{x \in E} e(x \cdot \xi),  
    \end{align*}
    From \eqref{eq:3.7} we obtain
    \begin{equation} \label{eq:3.8}
    |m_N^q(\xi)-s_N^q(\xi)| \lesssim \frac{1}{n},
    \end{equation}
    so it suffices to show that 
    \[
    |s_N^q(\xi)| \lesssim e^{- \frac{c\kappa_q(d,N)^q}{320} \min\big(\|\xi\|^2, \| \xi +1/2\|^2 \big)}.
    \]
    We decompose any $x \in \Z^d$ into $x=y(x)+z(x)$, where
    \[
z(x)_j= \begin{cases}
    x_j, \ \ \text{if $|x_j| \geq 2$} \\
    0, \ \ \text{otherwise}
\end{cases}
\]
and denote $Z(E)=\{z(x): x \in E \}$. Note that we have the following disjoint decomposition of the set $E$.
\[
E= \bigcup_{z \in Z(E)} \bigcup_{n/2<k \leq n-\|z\|_q^q} \Big(z+ \Big\{y \in \{-1,0,1\}^d: \supp(y) \cap \supp(z)= \emptyset, |\{i \in [d]: |y_i|=1 \}|=k\Big\}\Big),
\]
using the above we get that
\begin{equation} \label{eq:3.9}
 \begin{split} 
    s_N^q(\xi)&= \frac{1}{|B_N^q\cap \Z^d|} \sum_{x \in E} e(x \cdot \xi)= \frac{1}{|B_N^q\cap \Z^d|}\sum_{z \in Z(E)} e(z \cdot \xi) \sum_{n/2<k \leq n-\|z \|_q^q} \sum_{\substack{y \in \{-1,0,1\}^d, \\ |\{i \in [d]: |y_i|=1\}|=k, \\
    \supp(y) \cap \supp(z)= \emptyset}} e(y \cdot \xi)
 \end{split}   
\end{equation}
By Proposition \ref{prop:3.4} for any $z \in Z(E)$, $k \in \N$, such that $ n/2 < k \le n -\|z\|_q^q \le \frac{d-|\supp(z)|}{2}  $, let $J= [d] \setminus \supp(z)$, then we have
\begin{align*}
&\bigg|\sum_{\substack{y \in \{-1,0,1\}^d, \\ |\{i \in [d]: |y_i|=1\}|=k, \\
    \supp(y) \cap \supp(z)= \emptyset}} e(y \cdot \xi) \bigg|= |D_k^{J}| |\beta_k^{ J}(\xi)| \leq 2|D_k^{J}| e^{-\frac{ck}{2|J|} \min\big( \sum_{j \in J} \sin^2(\pi \xi_j), \sum_{j \in J} \cos^2(\pi \xi_j) \big)}  \\
& \le  2 e^{-\frac{cn}{4d} \min\big( \sum_{j \in [d] \setminus \supp(z)} \sin^2(\pi \xi_j), \sum_{j \in [d] \setminus \supp(z)} \cos^2(\pi \xi_j) \big)} \cdot 2^k \binom{d-|\supp(z)|}{k}.
\end{align*}
Plugging the above into \eqref{eq:3.9} we get
\begin{align*}
    |s_N^q(\xi)| &\le \frac{2}{|B_N^q\cap \Z^d|}\sum_{z \in Z(E)} \sum_{n/2<k \leq n-\|z \|_q^q} 2^k \binom{d-|\supp(z)|}{k} \exp\Big(-\frac{cn}{4d} \sum_{j \in [d] \setminus \supp(z)} \sin^2(\pi \xi_j) \Big) \\
    &+ \frac{2}{|B_N^q\cap \Z^d|}\sum_{z \in Z(E)}  \sum_{n/2<k \leq n-\|z \|_q^q} 2^k \binom{d-|\supp(z)|}{k} \exp\Big(-\frac{cn}{4d} \sum_{j \in [d] \setminus \supp(z)} \cos^2(\pi \xi_j) \Big).
\end{align*}
We will bound only the first term above, for the second summand treatment is analogous. Note that by symmetry invariance of the set $Z(E)$ and Lemma \ref{lem:2.6} (with $\delta_0=\delta_1=1/2$) one has
\begin{align*}
    &\frac{2}{|B_N^q\cap \Z^d|}\sum_{z \in Z(E)} \sum_{n/2<k \leq n-\|z \|_q^q} 2^k \binom{d-|\supp(z)|}{k} \exp\Big(-\frac{cn}{4d} \sum_{j \in [d] \setminus \supp(z)} \sin^2(\pi \xi_j) \Big) \\
    &=\frac{2}{|B_N^q\cap \Z^d|}\sum_{z \in Z(E)} \sum_{n/2<k \leq n-\|z \|_q^q} 2^k \binom{d-|\supp(z)|}{k} \frac{1}{d!} \sum_{\sigma \in \Sym(d)} \exp\Big(-\frac{cn}{4d} \sum_{j \in \sigma\big( [d] \setminus \supp(z)\big)} \sin^2(\pi \xi_j) \Big) \\
    &\le \frac{8}{|B_N^q\cap \Z^d|}\sum_{z \in Z(E)} \sum_{n/2<k \leq n-\|z \|_q^q} 2^k \binom{d-|\supp(z)|}{k}  e^{- \frac{cn}{320d} \| \xi \|^2}= 8 s_N^q(0)e^{- \frac{cn}{320d} \| \xi \|^2} \le 8 e^{- \frac{cn}{320d} \| \xi \|^2}.
\end{align*}
Similarly one can show that 
\[
\frac{2}{|B_N^q\cap \Z^d|}\sum_{z \in Z(E)}  \sum_{n/2<k \leq n-\|z \|_q^q} 2^k \binom{d-|\supp(z)|}{k} \exp\Big(-\frac{cn}{4d} \sum_{j \in [d] \setminus \supp(z)} \cos^2(\pi \xi_j) \Big) \le 8 e^{- \frac{cn}{320d} \| \xi +1/2\|^2}.
\]
Since $\frac{n}{d}=\kappa_q(d,N)^q$, the above combined with \eqref{eq:3.8} finished proof of \eqref{eq:3.4}. 
\par Now we will prove the inequality \eqref{eq:3.5}. As in the statement of the inequality assume that 
\[
\|\xi \| \leq \| \xi +1/2 \|.
\]
We define set $E$ and multiplier $s_N^q(\xi)$ as before.
Let 
\[
\phi_N^q(\xi)= \frac{1}{|B_N^q\cap \Z^d|}\sum_{z \in Z(E)}  \sum_{n/2<k \leq n-\|z \|_q^q} \sum_{\substack{y \in \{-1,0,1\}^d, \\ |\{i \in [d]: |y_i|=1\}|=k, \\
    \supp(y) \cap \supp(z)= \emptyset}} e(y \cdot \xi).
\]
From inequality $|\prod_{j=1}^da_j-1|\le \sum_{j=1}^d |a_j-1|$ ,valid if $\sup_{j}|a_j| \leq 1$, for any $z \in Z(E)$ we have
\begin{equation} \label{eq:3.10}
|e(z \cdot \xi)-1| \le \sum_{j=1}^d |e(z_j \xi_j)-1|=2 \sum_{j=1}^d \sin^2(z_j \xi_j) \leq 2 \sum_{j=1}^d \sin^2(\pi\xi_j) z_j^2.     
\end{equation}
At the end we have used inequality $|\sin(xt)| \le |x||\sin(t)|$, which holds for all $x \in \Z$ and $t \in \mathbb{R}$. Using the above, assumption on $\xi$,  \eqref{eq:3.9} and Proposition \ref{prop:3.4} we obtain
\begin{equation} \label{eq:3.11}
\begin{split}
    &|s_N^q(\xi)-\phi_N^q(\xi)| \le \frac{1}{|B_N^q\cap \Z^d|}\sum_{z \in Z(E)} \Big|e(z \cdot \xi)-1\Big| \sum_{n/2<k \leq n-\|z \|_q^q} \bigg| \sum_{\substack{y \in \{-1,0,1\}^d, \\ |\{i \in [d]: |y_i|=1\}|=k, \\
    \supp(y) \cap \supp(z)= \emptyset}} e(y \cdot \xi) \bigg|\\
   &\le\frac{4}{|B_N^q\cap \Z^d|}\sum_{z \in Z(E)}\sum_{j=1}^d \sin^2(\pi \xi_j) z_j^2 \sum_{n/2<k \leq n-\|z \|_q^q}
    2^k \binom{d-|\supp(z)|}{k} \exp\Big(-\frac{cn}{4d} \sum_{i \in [d] \setminus \supp(z)} \sin^2(\pi \xi_i) \Big) \\
    &+\frac{4}{|B_N^q\cap \Z^d|}\sum_{z \in Z(E)}\sum_{j=1}^d \sin^2(\pi \xi_j) z_j^2 \sum_{n/2<k \leq n-\|z \|_q^q}
    2^k \binom{d-|\supp(z)|}{k} \exp\Big(-\frac{cn}{4d} \sum_{i \in [d] \setminus \supp(z)} \cos^2(\pi \xi_i) \Big)
\end{split}
\end{equation}
We will investigate each of the terms above separately.
For the quarter of the first term by H\"older's inequality for exponent pair $(\frac{100}{99},100)$ and Corollary \ref{Cor:2.7} we have 
\begin{equation}
    \begin{split} \label{eq:3.12}
    &\frac{1}{|B_N^q\cap \Z^d|}\sum_{z \in Z(E)}\sum_{j=1}^d \sin^2(\pi \xi_j) z_j^2 \sum_{n/2<k \leq n-\|z \|_q^q}
    2^k \binom{d-|\supp(z)|}{k} \exp\Big(-\frac{cn}{4d} \sum_{i \in [d] \setminus \supp(z)} \sin^2(\pi \xi_i) \Big) \\
    &=\sum_{j=1}^d \sin^2(\pi \xi_j) \frac{1}{|B_N^q\cap \Z^d|}\sum_{z \in Z(E)}  \sum_{n/2<k \leq n-\|z \|_q^q}
    2^k \binom{d-|\supp(z)|}{k} z_j^2\exp\Big(-\frac{cn}{4d} \sum_{i \in [d] \setminus \supp(z)} \sin^2(\pi \xi_i) \Big) \\
    &\le \sum_{j=1}^d \sin^2(\pi \xi_j) \bigg( \frac{1}{|B_N^q\cap \Z^d|}\sum_{z \in Z(E)}  \sum_{n/2<k \leq n-\|z \|_q^q}
    2^k \binom{d-|\supp(z)|}{k} |z_j|^{200/99} \bigg)^{99/100} \\
    &\cdot
    \bigg( \frac{1}{|B_N^q\cap \Z^d|}\sum_{z \in Z(E)}  \sum_{n/2<k \leq n-\|z \|_q^q}
    2^k \binom{d-|\supp(z)|}{k} \exp\Big(-25\frac{cn}{d} \sum_{i \in [d] \setminus \supp(z)} \sin^2(\pi \xi_i) \Big) \bigg)^{1/100} \\
    &\lesssim
    e^{- \frac{cn}{10^6d} \| \xi\|^2} \sum_{j=1}^d \sin^2(\pi \xi_j) \bigg( \frac{1}{|B_N^q\cap \Z^d|}\sum_{z \in Z(E)}  \sum_{n/2<k \leq n-\|z \|_q^q}
    2^k \binom{d-|\supp(z)|}{k} z_j^4 \bigg)^{99/100}.
\end{split}
\end{equation}
By analogous arguments and assumption $\| \xi \| \le \| \xi+1/2 \|$ for the quarter of the second term of \eqref{eq:3.11} we have
\begin{equation}
    \begin{split}\label{eq:3.13}
    &\frac{1}{|B_N^q\cap \Z^d|}\sum_{z \in Z(E)}\sum_{j=1}^d \sin^2(\pi \xi_j) z_j^2 \sum_{n/2<k \leq n-\|z \|_q^q}
    2^k \binom{d-|\supp(z)|}{k} \exp\Big(-\frac{cn}{4d} \sum_{i \in [d] \setminus \supp(z)} \cos^2(\pi \xi_i) \Big) \\
    &\cdot
    e^{- \frac{cn}{10^6d} \| \xi+1/2\|^2} \sum_{j=1}^d \sin^2(\pi \xi_j)\bigg( \frac{1}{|B_N^q\cap \Z^d|}\sum_{z \in Z(E)}  \sum_{n/2<k \leq n-\|z \|_q^q}
    2^k \binom{d-|\supp(z)|}{k} z_j^4 \bigg)^{99/100} \\
    &\lesssim
    e^{- \frac{cn}{10^6d} \| \xi\|^2} \sum_{j=1}^d \sin^2(\pi \xi_j) \bigg( \frac{1}{|B_N^q\cap \Z^d|}\sum_{z \in Z(E)}  \sum_{n/2<k \leq n-\|z \|_q^q}
    2^k \binom{d-|\supp(z)|}{k} z_j^4 \bigg)^{99/100}.
\end{split}
\end{equation}
Notice that by the invariance of the set $\Big\{z \in Z(E): \|z \|_q^q \le n-k,  |\supp(z)|=u \Big\}$ under the action of $\Sym(d)$, and definition of the set $E$ we obtain
\begin{equation}
\begin{split} \label{eq:3.14}
    &\frac{1}{|B_N^q\cap \Z^d|}\sum_{z \in Z(E)} \sum_{n/2<k \leq n-\|z \|_q^q}
    2^k \binom{d-|\supp(z)|}{k} z_j^4\\
    &=\frac{1}{|B_N^q\cap \Z^d|}\sum_{n/2<k \leq n} \sum_{u=0}^d \cdot 2^k \binom{d-u}{k} \hspace{-0.4cm}  \sum_{\substack{z \in Z(E), \\ \|z \|_q^q \le n-k, \\ |\supp(z)|=u}}\hspace{-0.4cm}  z_j^4 =\frac{1}{|B_N^q\cap \Z^d|}\sum_{n/2<k \leq n} \sum_{u=0}^d \cdot 2^k \binom{d-u}{k} \frac{1}{d}\hspace{-0.4cm}  \sum_{\substack{z \in Z(E), \\ \|z \|_q^q \le n-k, \\ |\supp(z)|=u}}\hspace{-0.3cm} \|z\|_4^4 \\
    &\le \frac{ \kappa_q(d,N)^{q/22}n}{d} \frac{1}{|B_N^q\cap \Z^d|}\sum_{n/2<k \leq n} \sum_{u=0}^d \cdot 2^k \binom{d-u}{k}  \sum_{\substack{z \in Z(E), \\ \|z \|_q^q \le n-k, \\ |\supp(z)|=u}} 1= \kappa_q(d,N)^{23q/22} s_N^q(0).
    \end{split}
\end{equation}
Using \eqref{eq:3.11},\eqref{eq:3.12},\eqref{eq:3.13} and \eqref{eq:3.14} we get
\begin{equation} \label{eq:3.15}
\begin{split}
    &|s_N^q(\xi)-\phi_N^q(\xi)| \lesssim e^{- \frac{cn}{10^6d} \| \xi\|^2} \sum_{j=1}^d \sin^2(\pi \xi_j) \bigg( \frac{1}{|B_N^q\cap \Z^d|}\sum_{z \in Z(E)}  \sum_{n/2<k \leq n-\|z \|_q^q}
    2^k \binom{d-|\supp(z)|}{k} z_j^4 \bigg)^{99/100} \\
    &\lesssim e^{- \frac{cn}{10^6d} \| \xi\|^2} \|\xi \|^2 \kappa_q(d,N)^{\frac{23*99}{22*100}q} \lesssim  \kappa_q(d,N)^{\frac{23*99}{22*100}q-q} \lesssim  \kappa_q(d,N)^{q/50}.
\end{split}
\end{equation}
 Combining \eqref{eq:3.8} and \eqref{eq:3.15} we have
\begin{equation}
    \label{eq:3.16}
    |m_N^q(\xi)-\phi_N^q(\xi)| \lesssim \frac{1}{n}+ \kappa_q(d,N)^{q/50} .
\end{equation}
So in order to prove \eqref{eq:3.5} it suffices to show that
\begin{equation} \label{eq:3.17}
|\phi_N^q(\xi)-1| \lesssim  \kappa_q(d,N)^q \| \xi\|^2 + \frac{1}{n}.
\end{equation}
This follows from \eqref{eq:3.7} and Proposition \ref{prop:3.4}, indeed we have
\begin{align*}
&|\phi_N^q(\xi)-1|= \bigg|\phi_N^q(\xi)-s_N^q(0)-\frac{|B_N^q \cap \Z^d \setminus E|}{|B_N^q \cap \Z^d|}\bigg| \lesssim |\phi_N^q(\xi)-s_N^q(0)|+ \frac{1}{n} \\
&\le \frac{1}{|B_N^q\cap \Z^d|}\sum_{z \in Z(E)}  \sum_{n/2<k \leq n-\|z \|_q^q} \bigg|\sum_{\substack{y \in \{-1,0,1\}^d, \\ |\{i \in [d]: |y_i|=1\}|=k, \\
    \supp(y) \cap \supp(z)= \emptyset}} e(y \cdot \xi)-1 \bigg|+ \frac{1}{n} \\
    &\lesssim  \frac{1}{|B_N^q\cap \Z^d|}\sum_{z \in Z(E)}  \sum_{n/2<k \leq n-\|z \|_q^q} 2^k \binom{d-|\supp(z)|}{k} \frac{k}{d-|\supp(z)|}\sum_{j \in [d] \setminus \supp(z)} \sin^2(\pi \xi_j)+ \frac{1}{n}\\
    &\lesssim \frac{n}{d} \sum_{j=1}^d \sin^2(\pi \xi_j) s_N^q(0)+ \frac{1}{n} \le \kappa_q(d,N)^q \|\xi \|^2+ \frac{1}{n},
\end{align*}
above we've used the fact that $|\supp(z)| \le n \le d/2$.
This finished the proof \eqref{eq:3.17} and hence the proof of \eqref{eq:3.5}. 
\par To finish the proof of Proposition \ref{prop:3.5} it remains to prove \eqref{eq:3.6}, due to major similarities with \eqref{eq:3.5}, we will give less details in its proof. Assume that $\xi$ satisfies 
\[
\|\xi+1/2\|< \|\xi\|,
\]
define $\eta \in \T^d$ by $\eta_j=\xi_j+1/2$.
For $E$, $s_N^q(\xi)$ defined as earlier, notice that
\begin{equation*}
    s_N^q(\xi)=
    \frac{1}{|B_N^q\cap \Z^d|}\sum_{z \in Z(E)} (-1)^{\sum_{j=1}^d z_j}e(z \cdot \eta) \hspace{-0.5cm}\sum_{n/2<k \leq n-\|z \|_q^q} \hspace{-0.5cm}(-1)^k\hspace{-0.6cm}\sum_{\substack{y \in \{-1,0,1\}^d, \\ |\{i \in [d]: |y_i|=1\}|=k, \\
    \supp(y) \cap \supp(z)= \emptyset}} e(y \cdot \eta)
\end{equation*}
and let
\begin{equation*}
    \widetilde{\phi_N^q}(\xi)=
    \frac{1}{|B_N^q\cap \Z^d|}\sum_{z \in Z(E)} (-1)^{\sum_{j=1}^d z_j} \hspace{-0.5cm}\sum_{n/2<k \leq n-\|z \|_q^q} \hspace{-0.5cm}(-1)^k\hspace{-0.6cm}\sum_{\substack{y \in \{-1,0,1\}^d, \\ |\{i \in [d]: |y_i|=1\}|=k, \\
    \supp(y) \cap \supp(z)= \emptyset}} e(y \cdot \eta).
\end{equation*}
Using exactly the same method as in the proof of \eqref{eq:3.5} one can show that 
    $|s_N^q(\xi)-\widetilde{\phi_N^q}(\xi)| \lesssim \frac{1}{n}+ \kappa_q(d,N)^{q/50}.$
Combining this with \eqref{eq:3.8} we have $|m_N^q(\xi)- \widetilde{\phi_N^q}(\xi)| \lesssim \frac{1}{n}+ \kappa_q(d,N)^{q/50}.$

It remains to show that 
\begin{equation} \label{eq:3.18}
    \Big|\widetilde{\phi_N^q}(\xi)- \frac{1}{|B_N^q \cap \Z^d|} \sum_{x \in B_N^q \cap \Z^d} (-1)^{\sum_{j=1}^d x_j}\Big| \lesssim \kappa_q(d,N)^q \| \xi+1/2\|^2 + \frac{1}{n}.
\end{equation}
This follows from \eqref{eq:3.7} and Proposition \eqref{prop:3.4}, indeed similarly as before we have
\begin{align*}
&\Big|\widetilde{\phi_N^q}(\xi)- \frac{1}{|B_N^q \cap \Z^d|} \sum_{x \in B_N^q \cap \Z^d} (-1)^{\sum_{j=1}^d x_j}\Big| \leq 
\Big|\widetilde{\phi_N^q}(\xi)- \frac{1}{|B_N^q \cap \Z^d|} \sum_{x \in E} (-1)^{\sum_{j=1}^d x_j}\Big|+ \frac{|B_N^q \cap \Z^d \setminus E|}{|B_N^q \cap \Z^d|} \\
&\lesssim \frac{1}{|B_N^q\cap \Z^d|}\Bigg|\sum_{z \in Z(E)} (-1)^{\sum_{j=1}^d z_j} \hspace{-0.5cm}\sum_{n/2<k \leq n-\|z \|_q^q} \hspace{-0.5cm}(-1)^k\bigg(\hspace{-0.5cm}\sum_{\substack{y \in \{-1,0,1\}^d, \\ |\{i \in [d]: |y_i|=1\}|=k, \\
    \supp(y) \cap \supp(z)= \emptyset}} e(y \cdot \eta) -1 \bigg)\Bigg|+ \frac{1}{n}  \\
    &\lesssim 
     \frac{1}{|B_N^q\cap \Z^d|}\sum_{z \in Z(E)}  \sum_{n/2<k \leq n-\|z \|_q^q} \bigg|\sum_{\substack{y \in \{-1,0,1\}^d, \\ |\{i \in [d]: |y_i|=1\}|=k, \\
    \supp(y) \cap \supp(z)= \emptyset}} e(y \cdot \eta)-1 \bigg|+ \frac{1}{n} \\
    &\lesssim \kappa_q(d,N)^q \| \eta \|^2 s_N^q(0)+ \frac{1}{n} \le \kappa_q(d,N)^q \|\xi+1/2\|^2+ \frac{1}{n},
\end{align*}
this finishes the proof \eqref{eq:3.18}, hence the proof of \eqref{eq:3.6}, thus finishes the entire proof of Proposition \ref{prop:3.5}.
\end{proof}
\section{Conclusions}
In this section by using Proposition \ref{prop:3.5} we will conclude the proof of Theorem \ref{thm:1.1}.
We begin by defining two auxiliary multipliers 
\begin{Defn}
    For any $q \geq 1$, $d \in \N$, $t>0$ we define $\lambda_{t}^{1,q},\lambda_{t}^{2,q}: \T^d \to \mathbb{C}$ by following formulas
    \[
\lambda^{1,q}_t(\xi)=e^ {- \kappa_q(d,t)^q  \| \xi \|^2 },
\]
\[
\lambda^{2,q}_t(\xi)= \frac{1}{|B_t^q \cap \Z^d|} \Big(\sum_{x \in B_t^q \cap \Z^d} (-1)^{\sum_{i=1}^d x_i} \Big) e^ {- \kappa_q(d,t)^q  \| \xi +1/2\|^2 }. 
\]
\end{Defn}
Proof of the following theorem is a consequence of the theory of symmetric diffusion semigroups, see  \cite[p. 73]{Stein+1970} and \cite[Theorem 2.11]{NW}.
\begin{Thm} \label{thm:4.2}
    There exists a constant $C>0$ such that for all $q \geq 1, d \in \N$ and $f \in \ell^2(\Z ^d)$ the following inequality holds
    \[ 
    \Big\| \sup_{t>0} \big|\mathcal{F}^{-1}(\lambda_t^{1,q} \widehat{f})\big|\Big\|_{\ell^2(\Z^d)} \leq C \|f \|_{\ell^2(\Z^d)}.
    \]
\end{Thm}
Notice that $\lambda_t^{2,q}(\xi)=\frac{1}{|B_t^q \cap \Z^d|} \Big(\sum_{x \in B_t^q \cap \Z^d} (-1)^{\sum_{i=1}^d x_i} \Big)\lambda_t^{1,q}(\xi+\frac{1}{2})$, where $\xi+\frac{1}{2}=(\xi_1 + \frac{1}{2},..., \xi_d+ \frac{1}{2})$. Because of this we have
\[
\sup_{\|f \|_{\ell^2(\Z ^d)}=1} \Big\| \sup_{t>0} \big|\mathcal{F}^{-1}(\lambda_t^{2,q} \widehat{f})\big|\Big \|_{\ell^2(\Z^d)}= \sup_{\|f \|_{\ell^2(\Z ^d)}=1} \Big\| \sup_{t >0} \big|\mathcal{F}^{-1}(\lambda_t^{1,q} \widehat{f})\big|\Big \|_{\ell^2(\Z^d)}.
\]
It turns out multipliers $\lambda_N^{1,q}, \lambda^{2,q}_N$ are good approximations to $m_N^q$, this is stated below.
\begin{Prop} \label{prop:4.3}
Let $q \geq 1$, $d,N \in \N$ satisfy $\kappa_q(d,N) \le e^{-12/q}$. Then for every $\xi \in \T^d$ we have \\
\begin{enumerate}
    \item If $\|\xi \| \leq \|\xi+1/2 \|$, then
    \begin{equation*}
    |m_N^q(\xi)-\lambda^{1,q}_N(\xi)| \lesssim \min \bigg( \kappa_q(d,N)^q \|\xi \|^2,  \Big(\kappa_q(d,N)^q \|\xi \|^2 \Big)^{-1} \bigg)+\kappa_q(d,N)^{q/50}+\frac{1}{N^{q}}.
    \end{equation*}

    \item If $\|\xi +1/2\| < \|\xi \|$, then
    \begin{equation*}
    |m_N^q(\xi)-\lambda^{2,q}_N(\xi)| \lesssim \min \bigg( \kappa_q(d,N)^q \|\xi +1/2\|^2,  \Big(\kappa_q(d,N)^q \|\xi +1/2\|^2 \Big)^{-1} \bigg)+\kappa_q(d,N)^{q/50}+\frac{1}{N^{q}}.
    \end{equation*}
\end{enumerate}
\end{Prop}
Proof of the proposition above is a simple consequence of Proposition \ref{prop:3.5} together with elementary inequalities $1-e^{-x} \le x$, $e^{-x} \le x^{-1}$ for $x>0$.
Using Proposition \ref{prop:4.3} and a standard square function argument we can finally prove Theorem \ref{thm:1.4}.
\begin{proof}[Proof of Theorem \ref{thm:1.4}]
Take any $ q\geq 1$, $d \in \N$ and $f \in \ell^2(\Z^d)$. We decompose $f=f_1+f_2$, where
\[
\supp(\widehat{f_1}) \subseteq \{ \xi \in \T^d : \|\xi\| \leq \|\xi +1/2\|) \},
\quad
\supp(\widehat{f_2}) \subseteq \{ \xi \in \T^d : \|\xi\| > \|\xi +1/2\| \}.
\]
We will prove the following
\begin{equation} \label{eq:4.1}
    \Big\| \sup_{t \in \mathbb{D},  t \leq e^{-\frac{12}{q}}d^{1/q}} |\mathcal{M}_t^q f| \Big\|_{\ell^2(\mathbb{Z}^d)} \lesssim \| f \|_{\ell^2(\mathbb{Z}^d)},
\end{equation}
since $\mathcal{M}_t^q$ has $\ell^2$ norm at most $1$ and we have $|\D \cap (e^{-\frac{12}{q}}d,d]) | \le |\D \cap (e^{-12}d,d]) | \le \frac{12}{\log 2}+1 $, to establish Theorem \ref{thm:1.4} is suffices to prove \eqref{eq:4.1}. In fact we will show that for $i=1,2$ we have
\begin{equation} \label{eq:4.2}
    \Big\| \sup_{t \in \mathbb{D}, t \leq e^{-\frac{12}{q}}d^{1/q}} |\mathcal{M}_t^q f_i| \Big\|_{\ell^2(\mathbb{Z}^d)} \lesssim \| f_i \|_{\ell^2(\mathbb{Z}^d)}.
\end{equation}
We will consider only the case $i=1$, for $i=2$ proof is very similar. By Theorem \ref{thm:4.2} we have
\begin{equation} \label{eq:4.3}
    \begin{split}
         &\Big\| \sup_{t \in \mathbb{D}, t \leq e^{-\frac{12}{q}}d^{1/q}} |\mathcal{M}_t^q f_1| \Big\|_{\ell^2(\mathbb{Z}^d)} \le \Big\| \sup_{t \in \mathbb{D}, t \leq e^{-\frac{12}{q}}d^{1/q}} |\mathcal{F}^{-1}(\lambda_t^{1,q} \widehat{f_1})| \Big\|_{\ell^2(\mathbb{Z}^d)} \\
         &+ \Big\| \sup_{t \in \mathbb{D}, t \leq e^{-\frac{12}{q}}d^{1/q}} |\mathcal{F}^{-1}\big((m_t^q-\lambda_t^{1,q} )\widehat{f_1}\big)| \Big\|_{\ell^2(\mathbb{Z}^d)} \lesssim \|f_1 \|_{\ell^2(\Z^d)}+ \Big\| \sup_{t \in \mathbb{D}, t \leq e^{-\frac{12}{q}}d^{1/q}} |\mathcal{F}^{-1}\big((m_t^q-\lambda_t^{1,q} )\widehat{f_1}\big)| \Big\|_{\ell^2(\mathbb{Z}^d)}.
    \end{split}
\end{equation}
For the second term of the above we use a standard square function argument, i.e. Parseval's identity together with Proposition \ref{prop:4.3} and assumptions on $\supp(\widehat{f_1})$. We get
\begin{align*}
&\Big\| \sup_{t \in \mathbb{D}, t \leq e^{-\frac{12}{q}}d^{1/q}} |\mathcal{F}^{-1}\big((m_t^q-\lambda_t^{1,q} )\widehat{f_1}\big)| \Big\|_{\ell^2(\mathbb{Z}^d)} \leq 
\Big\| \Big( \sum_{t \in \mathbb{D}, t \leq e^{-\frac{12}{q}}d^{1/q}} |\mathcal{F}^{-1}\big((m_t^q-\lambda_t^{1,q} )\widehat{f_1}\big)|^2 \Big)^{1/2}\Big\|_{\ell^2(\mathbb{Z}^d)} \\
&= \Big(\hspace{-0.5cm} \sum_{t \in \mathbb{D}, t \leq e^{-\frac{12}{q}}d^{1/q}} \Big\|\mathcal{F}^{-1}\big((m_t^q-\lambda_t^{1,q} )\widehat{f_1}\big)\Big\|_{\ell^2(\mathbb{Z}^d)}^2 \Big)^{1/2}= 
\Big(\hspace{-0.5cm}\sum_{t \in \mathbb{D}, t \leq e^{-\frac{12}{q}}d^{1/q}} \int_{\T^d} |m_t^q(\xi)-\lambda_t^{1,q}(\xi)|^2|\widehat{f_1}(\xi)|^2 d\xi \Big)^{1/2} \\
&\lesssim \bigg(  \int_{\T^d} |\widehat{f_1}(\xi)|^2 
\cdot \bigg(\hspace{-0.5cm}\sum_{t \in \mathbb{D}, t \leq e^{-\frac{12}{q}}d^{1/q}} \hspace{-0.5cm}\min\bigg( \kappa_q(d,t)^q\|\xi \|^2,  \Big(\kappa_q(d,t)^q\|\xi \|^2 \Big)^{-1} \bigg)^2 + \hspace{-0.7cm}\sum_{t \in \D, t \leq e^{-\frac{12}{q}}d^{1/q}} \hspace{-0.7cm}\big(\frac{1}{t^{2q}}+ \kappa_q(d,t)^{q/25} \big)  \bigg)d\xi \bigg)^{1/2} \\
&\lesssim \bigg(  \int_{\T^d} |\widehat{f_1}(\xi)|^2 \bigg(\sum_{k=0}^{\infty} \min\bigg( \frac{2^{kq}}{d}\|\xi \|^2,  \Big(\frac{2^{kq}}{d}\|\xi \|^2 \Big)^{-1} \bigg)^2 + 1+ \sum_{k=0}^{ \lfloor \log_2(e^{-\frac{12}{q}}d^{1/q}) \rfloor} \frac{2^{kq/25}}{d} \bigg) d\xi \bigg)^{1/2} \\
&\lesssim  \bigg(  \int_{\T^d} |\widehat{f_1}(\xi)|^2 d\xi \bigg)^{1/2}=\|\widehat{f_1} \|_{L^2(\T^d)}= \|f_1 \|_{\ell^2(\Z^d)},
\end{align*}
above we have used the fact that for any $a \geq 2$ and $x \ge 0$ we have
$\sum_{k=0}^\infty \min (a^k x, a^{-k} x^{-1}) \lesssim 1,$
where the implied constant is universal. The above with \eqref{eq:4.3} establishes \eqref{eq:4.2} for $i=1$, as mentioned earlier the case $i=2$ is basically the same, but with $\lambda_t^{1,q}$ replaced by $\lambda_t^{2,q}$, so we skip its proof. Using triangle inequality, \eqref{eq:4.2} and Parseval's formula we have
\begin{align*}
    &\Big\| \sup_{t \in \mathbb{D}, t \leq e^{-\frac{12}{q}}d^{1/q}} |\mathcal{M}_t^q f| \Big\|_{\ell^2(\mathbb{Z}^d)} 
    \leq \Big\| \sup_{t \in \mathbb{D}, t \leq e^{-\frac{12}{q}}d^{1/q}} |\mathcal{M}_t^q f_1| \Big\|_{\ell^2(\mathbb{Z}^d)}+ \Big\| \sup_{t \in \mathbb{D}, t \leq e^{-\frac{12}{q}}d^{1/q}} |\mathcal{M}_t^q f_2| \Big\|_{\ell^2(\mathbb{Z}^d)} \\
    &\lesssim  \|f_1\|_{\ell^2(\Z^d)}+\|f_2\|_{\ell^2(\Z^d)}= 
   \|\widehat{f_1}\|_{L^2(\T^d)}+\|\widehat{f_2}\|_{L^2(\T^d)} \lesssim \|\widehat{f} \|_{L^2(\T^d)}= \| f \|_{\ell^2(\Z^d)},
\end{align*}
this concludes the proof of \eqref{eq:4.1} and hence the proof of Theorem \ref{thm:1.4}.
\end{proof}

\normalsize
\subsection*{Acknowledgements}
This research was funded by the National Science Centre, Poland, grant Sonata Bis
2022/46/E/ST1/00036. For the purpose of Open Access the authors have applied a CC BY public copyright licence
to any Author Accepted Manuscript (AAM) version arising from this submission


\begin{thebibliography}{1}


\bibitem[1]{balls} J. Bourgain, M. Mirek, E. Stein and B. Wróbel,
\emph{On Discrete Hardy--Littlewood Maximal Functions over the Balls in {$\mathbb{Z}^d$}: dimension-free estimates},
Geometric Aspects of Functional Analysis. Lecture Notes in Mathematics 2256, Springer, 2020, pages 127--169.


\bibitem[2]{cubes} J. Bourgain, M. Mirek, E. Stein and B. Wróbel,
\emph{ Dimension-Free Estimates for Discrete Hardy-Littlewood Averaging Operators Over the Cubes in {$\mathbb{Z}^d$}},
Amer. J. Math. 141, no.3, 2019, 857--905.

\bibitem[3]{contANDdiscr} J. Bourgain, M. Mirek, E.M. Stein, B. Wróbel,
\emph{On the Hardy–Littlewood Maximal Functions in High Dimensions: Continuous and Discrete Perspective.},
Geometric Aspects of Harmonic Analysis. Springer INdAM Series,vol 45. Springer, 107--148

\bibitem[4]{HL-cont} L. Deleaval, O. Guédon and B. Maurey,
\emph{Dimension free bounds for the Hardy–Littlewood maximal operator associated to convex sets.},  
Ann. Fac. Sci. Toulouse Math. (6) 27, no.1, 2018,1--198.

\bibitem[5]{v010a003} A. W. Harrow, A. Kolla and L. J. Schulman,
\emph{Dimension-Free $L_2$ Maximal Inequality for Spherical Means in the Hypercube},
Theory of Computing 10, no. 3, 2014, 55-75.

\bibitem[6]{Bq} D. Kosz, M. Mirek, P. Plewa and B. Wróbel,
\emph{Some remarks on dimension-free estimates for the discrete Hardy—Littlewood maximal functions.},
Israel J. Math 254, 2023, 1--38.

\bibitem[7]{412678}V. I. Levenshtein,
\emph{Krawtchouk polynomials and universal bounds for codes and designs in Hamming spaces},
IEEE Trans. Inform. Theory 41, no. 5, 1995, 1303--1321
\bibitem[8]{MiSzWr} M.\ Mirek, T.Z.\ Szarek, B.\ Wr\'obel,
\emph{Dimension-Free Estimates for the Discrete Spherical Maximal Functions}, International Mathematics Research Notices, Volume 2024, Issue 2, January 2024, Pages 901–963. doi: 10.1093/imrn/rnac329

\bibitem[9]{ja} J. Niksiński,
\emph{Dimension-free estimates on $l^2(\Z^d)$ for a discrete dyadic maximal function over $l^1$ balls: small scales},
Colloquium Mathematicum 175 (2024), 37--54.

\bibitem[10]{NW} J. Niksiński, B. Wróbel, 
\emph{Dimension-free estimates for discrete maximal functions and lattice points in high-dimensional spheres and balls with small radii},  preprint 2025, arXiv:2503.16952

\bibitem[11]{Stein+1970} E. Stein,
\emph{Topics in Harmonic Analysis Related to the Littlewood-Paley Theory. (AM-63), Volume 63},
Princeton University Press, 1970.

\bibitem[12]{DLMF}
NIST Digital Library of Mathematical Functions. https://dlmf.nist.gov/, Release 1.2.4 of 2025-03-15. F. W. J. Olver, A. B. Olde Daalhuis, D. W. Lozier, B. I. Schneider, R. F. Boisvert, C. W. Clark, B. R. Miller, B. V. Saunders, H. S. Cohl, and M. A. McClain, eds.
\end{thebibliography}
\end{document}